\documentclass[12pt,a4paper]{article}
\usepackage{amssymb,amsmath,amsthm}
\usepackage{latexsym}

\newtheorem{theorem}{Theorem}
\newtheorem{lemma}{Lemma}
\newtheorem{proposition}{Proposition}
\newtheorem{corollary}{Corollary}
\newtheorem{remark}{Remark}
\numberwithin{equation}{section}
\numberwithin{theorem}{section}
\numberwithin{lemma}{section}
\numberwithin{proposition}{section}
\numberwithin{corollary}{section}
\numberwithin{remark}{section}

\begin{document}
\title{Decay estimates of gradient of a generalized Oseen evolution operator
arising from time-dependent rigid motions in exterior domains}
\author{Toshiaki Hishida\thanks{
Partially supported by the Grant-in-aid for Scientific Research 18K03363
from JSPS} \\
Graduate School of Mathematics \\
Nagoya University \\
Nagoya 464-8602, Japan \\
\texttt{hishida@math.nagoya-u.ac.jp}}
\date{}
\maketitle
\begin{abstract}
Let us consider the motion of a viscous incompressible fluid 
past a rotating rigid body in 3D,
where the translational and angular velocities of the body 
are prescribed but time-dependent.
In a reference frame attached to the body, we have the Navier-Stokes system 
with the drift and
(one half of the) Coriolis terms in a fixed exterior domain.
The existence of the evolution operator 
$T(t,s)$ in the space $L^q$ generated by the linearized non-autonomous
system was proved by Hansel and Rhandi
\cite{HR14} and the large time behavior of $T(t,s)f$ in $L^r$ for
$(t-s)\to\infty$ was then developed by the present author \cite{Hi18}
when $f$ is taken from $L^q$ with $q\leq r$.
The contribution of the present paper concerns such 
$L^q$-$L^r$ decay estimates of
$\nabla T(t,s)$ with optimal rates,
which must be useful for the study of stability/attainability of the 
Navier-Stokes flow
in several physically relevant situations.
Our main theorem completely recovers the $L^q$-$L^r$ estimates for the 
autonomous case
(Stokes and Oseen semigroups, those semigroups with rotating effect)
in 3D exterior domains, which were established by
\cite{I}, \cite{MSol}, \cite{KS}, \cite{HiS} and \cite{Shi08}.
\end{abstract}

\section{Introduction}
\label{intro}

This paper is the continuation of the previous study \cite{Hi18}
on large time behavior of a generalized Oseen
evolution operator $T(t,s)$, which is the solution operator
$u(\cdot,s)=f\mapsto u(\cdot,t)$ to the initial value problem for the linear
non-autonomous system
\begin{equation}
\begin{split}
\partial_tu
&=\Delta u+(\eta(t)+\omega(t)\times x)\cdot\nabla u-\omega(t)\times u-\nabla p,
 \\
\mbox{div $u$}&=0, \\
u|_{\partial D}&=0, \\
u&\to 0\quad\mbox{as $|x|\to\infty$}, \\
u(\cdot,s)&=f,
\end{split}
\label{IVP}
\end{equation}
in $D\times (s,\infty)$, where $D$ is an exterior domain in $\mathbb R^3$
with $C^{1,1}$-boundary $\partial D$,
$\{u(x,t),p(x,t)\}$ with $u=(u_1,u_2,u_3)^\top$
is the pair of unknowns which are the velocity vector field and 
pressure of a viscous fluid, respectively, while
the solenoidal vector field $f(x)=(f_1,f_2,f_3)^\top$ is a given initial 
velocity at initial time $s\geq 0$
and $\{\eta(t),\omega(t)\}\in\mathbb R^{3\times 2}$ will be explained soon.
Here and in what follows,
$(\cdot)^\top$ stands for the transpose of vectors or matirices.
Problem \eqref{IVP} is a linearized system for the Navier-Stokes
problem modeling a viscous incompressible flow past an obstacle
$\mathbb R^3\setminus D$ (rigid body) that moves in a prescribed way.
One usually makes a transformation of variables
in order to reduce the problem to an equivalent one over the fixed domain
in a frame attached to the obstacle, see Galdi \cite{Ga02}
for details.
Then the resulting system is \eqref{IVP} (with $s=0$) in which the LHS
of the equation of motion should be
replaced by $\partial_tu+u\cdot\nabla u$ and the fluid velocity attains
the rigid motion $\eta+\omega\times x$ (no-slip condition) at the
boundary $\partial D$, where
$\eta(t)$ and $\omega(t)$ respectively
denote the translational and angular velocities of the rigid body
(after the transformation mentioned above).
This paper develops methods of analyzing 
the large time behavior of $T(t,s)$ for $(t-s)\to\infty$ 
when both translational and angular velocities are time-dependent.
Our conditions on this dependence are
\begin{equation}
\eta,\,\omega\in C^\theta([0,\infty); \mathbb R^3)
\cap L^\infty(0,\infty; \mathbb R^3)
\label{rigid}
\end{equation}
with some $\theta\in (0,1)$, which are the same as 
in the previous study \cite{Hi18}.

The well-posedness of \eqref{IVP}, that is, generation of the evolution
operator $\{T(t,s)\}_{t\geq s\geq 0}$ in the space $L^q$ for $1<q<\infty$
was successfully proved by Hansel and Rhandi \cite{HR14} under
the condition
\begin{equation}
\eta,\,\omega\in C_{loc}^\theta([0,\infty); \mathbb R^3)
\label{loc-hoelder}
\end{equation}
with some $\theta\in (0,1)$.
It is reasonable not to need the global behavior \eqref{rigid}
just for the well-posedness of \eqref{IVP} and for regularity of the solution.
They also derived a remarkable $L^q$-$L^r$ smoothing action near the
initial time, that is,
\begin{equation}
\|T(t,s)f\|_r\leq C(t-s)^{-(3/q-3/r)/2}\|f\|_q
\label{LqLr}
\end{equation}
\begin{equation}
\|\nabla T(t,s)f\|_r\leq C(t-s)^{-(3/q-3/r)/2-1/2}\|f\|_q
\label{LqLr-grad}
\end{equation}
for $0\leq s<t\leq {\cal T}$ and $1<q\leq r<\infty$
with some constant $C>0$ that depends on ${\cal T}\in (0,\infty)$,
where $\|\cdot\|_q$ denotes the norm of the space $L^q(D)$.
Later on, the present author \cite{Hi18} has developed the 
$L^q$-$L^r$ decay estimate of $T(t,s)$, namely,
\eqref{LqLr} for all $t>s\geq 0$ and $1<q\leq r<\infty$ with some constant
$C>0$ independent of $(t,s)$.
A duality argument is one of ingredients of the proof, so that the
$L^q$-$L^r$ estimate of the adjoint evolution operator $T(t,s)^*$
has been also deduced in \cite{Hi18}
simultaneously with \eqref{LqLr}.
Note that the adjoint $T(t,s)^*$ is the solution operator
$v(\cdot,t)=g\mapsto v(\cdot,s)$ of the backward problem for the adjoint
system subject to the final condition at $t>0$,
see \eqref{back} below.
However, the decay estimate of $\nabla T(t,s)$
with optimal rate has remained open (\cite[Remark 2.1]{Hi18}).

The purpose of the present paper is to develop the gradient estimate
of the evolution operator for $(t-s)\to\infty$.
Our main theorem 
(Theorem \ref{main}, particularly the first assertion) provides us with
\eqref{LqLr-grad} for all $t>s\geq 0$ and $1<q\leq r\leq 3$. 
The rate of decay of $\nabla T(t,s)$ 
for the other case $1<q\leq r\in (3,\infty)$
is also discussed and it is given by
$(t-s)^{-3/2q}$.
In addition, we obtain the $L^q$-$L^\infty$ decay estimate of
$T(t,s)$ as well, that is, \eqref{LqLr} with $1<q<r=\infty$
for all $t>s\geq 0$.
Our theorem completely recovers the $L^q$-$L^r$ estimates for the
autonomous case developed by \cite{I}, \cite{MSol}
(both for the Stokes semigroup $\eta=\omega=0$),
\cite{KS}, \cite{ES04}, \cite{ES05}
(those three for the Oseen semigroup with constant $\eta\neq 0,\,\omega=0$),
\cite{HiS} (semigroup with constant $\omega\neq 0,\, \eta=0$) and
\cite{Shi08} (semigroup with constants $\eta\neq 0,\,\omega\neq 0$).
Therefore, analysis in this paper can be regarded as a unified approach 
not only for all the cases of uniform rigid motions
but for several cases of time-dependent ones.
Our result cannot be improved in general because
Maremonti and Solonnikov \cite{MSol} and the present author \cite{Hi11}
observed that the rate of decay of $\nabla T(t,s)$
in our theorem is optimal when $\eta=\omega=0$
(case of the Stokes semigroup).
Nevertheless, there might be a chance of improvement when $\eta\neq 0$;
for further discussion about the optimality, see Remark \ref{optimal}.

In view of the celebrated paper \cite{Ka} by Kato,
it is clear that we have several
applications of the complete $L^q$-$L^r$ estimates
\eqref{LqLr}--\eqref{LqLr-grad} for all $t>s\geq 0$
obtained in this paper.
In \cite{Hi18} (see also \cite{Hi19} for further development)
the present author has proposed 
a new way of constructing a unique Navier-Stokes flow
globally in time by use only of \eqref{LqLr}
combined with the energy relation (see \cite[Lemma 5.1]{Hi18}), 
but the solution constructed
in such a way possesses less information about the large time behavior;
in fact, an improvement of Theorem 5.1 of \cite{Hi18} 
by using \eqref{LqLr-grad} with $r=3$ for all $t>s\geq 0$ is obvious.
Since the same estimate for the adjoint $T(t,s)^*$ is
available in the Lorentz spaces as well, 
see \eqref{adj-Lorentz} in Theorem \ref{adj} below,
we must have even more applications 
with the aid of interpolation technique developed by Yamazaki \cite{Y}.
Once we have \eqref{adj-Lorentz},
his insight brings us the sharp estimate \eqref{adj-int}, which is quite
useful to study the stability/attainability of several physically relevant
background flows (not only steady flow but also time-dependent flows
such as time-periodic one) being in the scale-critical Lorentz space
$L^{3,\infty}$ (weak-$L^3$ space).
This is indeed the case if, for instance, the obstacle is
purely rotating or at rest without translation,
where the optimality of the decay rate $|x|^{-1}$ for
generic flow is interpreted in terms of asymptotic structure at infinity, see
\cite{KSv}, \cite{FH11}, \cite{FGK} and \cite{Hi-hb}.
Several applications of our main theorems will be discussed elsewhere.
Let us just mention, as one of them, a problem of attainablity of 
a (small) steady flow
around a rigid body rotating from rest
(that was raised by \cite[Section 6]{Hi13}).
This is called the starting problem and was proposed first by
Finn \cite{Fi65} in the case when the rotation 
was replaced by translation of the body.
Finn's problem was successfully solved by Galdi, Heywood and Shibata \cite{GHS}
by making use of the $L^q$-$L^r$ estimate of the Oseen semigroup \cite{KS},
see also \cite{HiM} for further contributions,
however, the same approach with the aid of the $L^q$-$L^r$ estimate due
to \cite{HiS} no longer works for the question above because of
unbounded coefficient $\omega\times x$ of the drift term.
The right approach seems to be use of the results obtained here
for the non-autonomous system,
see \cite{Ta-new}.
Another application of our theorems would be the attainability of a
time-periodic flow arising from time-periodic translation 
with zero average (oscillation like back and forth),
whose existence has been recently proved by Galdi \cite{Ga-new}.

The proof of our main theorem 
consists of two stages:
One is the so-called local energy decay estimates over a bounded domain 
$D\cap B_R$ (near the obstacle),
see Propositions \ref{1st-LED} and \ref{2nd-LED},
the other is a decay estimate outside $B_R$ (near inifinity), 
where $B_R$ denotes the open ball centered at 
the origin with radius $R>0$.
Indeed this combination itself was adopted by several authors
(\cite{I}, \cite{KS}, \cite{ES05}, \cite{HiS} for 3D, \cite{DS-1}, 
\cite{DS-2}, \cite{Hi16}, \cite{Mae} for 2D)
for the autonomous case,
but what is new is to deduce the former without spectral
analysis.
In fact, our assumption \eqref{rigid} is too general
(without any specific structure such as time-periodicity)
to carry out the spectral analysis.
Note, however, that analysis of the resolvent near $\lambda=0$ 
is the essential and hard step
for the autonomous case in the literature above, where $\lambda$ denotes 
the spectral parameter.
We also refer the readers to a recent work \cite{Shi-aa} 
on the autonomous case by Shibata, who has developed even more
in the resolvent side to furnish the $L^q$-$L^r$ decay estimates.
In this paper, 
\eqref{LqLr} for all $t>s\geq 0$ plays a role to obtain 
the local energy decay estimates
(note that it is the opposite way to the argument in the literature mentioned 
above in which
\eqref{LqLr} was a conclusion of the local energy decay estimates),
but such estimates of $\nabla T(t,s)$ are not enough since we 
have to control the behavior
of the pressure at the other stage of deduction of decay estimates 
near infinity.
The natural idea is to analyze the asymptotic behavior,
both for $(t-s)\to\infty$ and for $(t-s)\to 0$,
of the temporal derivative $\partial_tT(t,s)$
in the Sobolev space of order $(-1)$ over the bounded domain $D\cap B_R$.
To this end, we need to develop more analysis of regularity of the
evolution operator $T(t,s)$, see Proposition \ref{weak},
than the one done by Hansel and Rhandi \cite{HR14}.
Analysis of $\partial_tT(t,s)$ is in fact very nontrivial since the 
corresponding autonomous operator
is no longer generator of an analytic semigroup in the space $L^q$
unless $\omega=0$,
see \cite{Hi99}, \cite{FN10} and the references therein, and it can be
regarded as a substitution of Section 5 of \cite{HiS} for the autonomous
case (semigroup with constant $\omega\neq 0$),
in which the authors made full use of precise behavior of parametrix of
the resolvent with respect to the spectral parameter.

It is worth while summarizing the method developed in the present
paper together with the previous study \cite{Hi18}.
The clue at the beginning toward analysis of large time
behavior of \eqref{IVP} would be:
\smallskip

(i) $L^q$-$L^r$ estimates \eqref{LqLr-whole}
for the same system in the whole space;

(ii) energy relations \cite[(2.15), (2.23)]{Hi18}
for $T(t,s)$ and its adjoint;
\smallskip

\noindent
both of which are clear because the equation in \eqref{IVP}
is derived only from the transformation of variables concerning (i)
and because the additional terms arising from this transformation
are skew-symmetric concerning (ii).
Those are fine, however, we would say that the only fine things
for \eqref{IVP} are them.
Note that, except for (ii), one does not have useful higher energy
estimates (which play an important role in \cite{MSol} for the
Stokes semigroup) unless $\eta=\omega=0$.
In \cite{Hi18} some devices by use of the energy (ii)
enable us to show the uniform boundedness
of $T(t,s)$ and $T(t,s)^*$ in $L^r$ with $r\in (2,\infty)$ 
by duality argument with the aid of (i) via cut-off procedure.
With this at hand, the deduction of \eqref{LqLr} for all
$t>s\geq 0$ can be reduced to computations of a differential inequality
\cite[Lemma 4.1, Lemma 4.2]{Hi18}.
And then, in this paper,
\eqref{LqLr} combined with a detailed analysis of $\partial_tT(t,s)$
leads us to \eqref{LqLr-grad} for all $t>s\geq 0$ and $1<q\leq r\leq 3$
as explained in the previous paragraph.
To sum up, along the approach proposed in both papers, once we have
(i) and (ii) above, we are able to deduce the large time behavior
of $\nabla^jT(t,s)$ with $j=0,1$ in 3D exterior domains.
In more involved 2D case, however, the method developed in 
\cite{Hi18} unfortunately does not work well, 
see \cite[Remark 4.1]{Hi18} for the difficulties.
Concerning the $L^q$-$L^r$ estimate for the autonomous case 
in 2D exterior domains,
we refer to \cite{DS-1}, \cite{DS-2}, \cite{MSol} (for the Stokes semigroup)
and \cite{Hi16}, \cite{Mae}
(for the Oseen semigroup, 
where the latter is a significant refinement of the former).
For the case of rotating obstacle, the desired decay property 
has still remained open in 2D even if $\omega\neq 0$ is a constant vector.

This paper is organized as follows.
In the next section,
after summarizing the knowledge from \cite{HR14} and \cite{Hi18},
we present the main theorems.
We need further analysis of the same system in the whole space and the one 
in bounded domains, which are not covered by the literature.
They are performed in Section \ref{whole} 
and Section \ref{interior}, respectively.
Along the way of constructing the evolution
operator due to \cite{HR14}, in Section \ref{regularity},
we develop more analysis of its regularity, in particular, smoothing rate 
as well as justification of
the temporal derivative $\partial_tT(t,s)f$ for general solenoidal vector 
field $f$ being in the space $L^q$.
Local energy decay estimates of the evolution operator near the
obstacle are established in Section \ref{local}.
The final section is devoted to completion of the proof
of the main theorems by showing the decay estimate of the
evolution operator near spatial infinity.

\section{Results}
\label{result}

Let us begin with introducing notation.
Given two vector fields $u$ and $v$, we denote by $u\otimes v$
the matrix $(u_iv_j)$.
Let $A=(A_{ij}(x))$ be a $3\times 3$ matrix-valued function, 
then the vector field
$\mbox{div $A$}$ is defined by
$(\mbox{div $A$})_i=\sum_j\partial_{x_j}A_{ij}$.
By following this rule, the drift and Coriolis terms in \eqref{IVP} can 
be expressed as
\[
(\eta+\omega\times x)\cdot\nabla u
=\mbox{div $\big[u\otimes (\eta+\omega\times x)\big]$}, \qquad
\omega\times u
=\mbox{div $\big[(\omega\times x)\otimes u\big]$},
\]
the latter of which follows from $\mbox{div $u$}=0$.
Those expressions appear in 
\eqref{weak-eqn-wh}, \eqref{weak-eqn-bdd} and \eqref{weak-eqn} below.

Given a domain $G\subset \mathbb R^3$, $q\in [1,\infty]$ and integer $k\geq 0$,
the standard Lebesgue and Sobolev spaces are denoted by
$L^q(G)$ and by $W^{k,q}(G)$.
We abbreviate the norm $\|\cdot\|_{q,G}=\|\cdot\|_{L^q(G)}$ and even
$\|\cdot\|_q=\|\cdot\|_{q,D}$, where $D$ is the exterior domain
under consideration with $C^{1,1}$-boundary $\partial D$.

Throughout this paper, we fix a number
$R_0>0$ so large that
\begin{equation}
\mathbb R^3\setminus D\subset B_{R_0},
\label{obstacle}
\end{equation}
where $B_R$ denotes the open ball centered at the origin
with radius $R>0$.
We set $D_R=D\cap B_R$ for $R\in [R_0,\infty)$.

The class $C_0^\infty(G)$ consists of all $C^\infty$ functions
with compact support in $G$, then $W^{k,q}_0(G)$ denotes
the completion of $C_0^\infty(G)$ in $W^{k,q}(G)$,
where $k>0$ is an integer.
We set $W^{-1,q}(G)=W^{1,q^\prime}_0(G)^*$, where
$1/q^\prime+1/q=1$ and $q\in (1,\infty)$.
By $\langle\cdot,\cdot\rangle_G$ we denote various
duality pairings over the domain $G$.
In what follows we adopt the same symbols for denoting 
scalar and vector (even tensor)
function spaces as long as there is no confusion.

Let $X_1$ and $X_2$ be two Banach spaces.
Then ${\cal L}(X_1,X_2)$ stands for the Banach space consisting of all
bounded linear operators from $X_1$ into $X_2$.
We simply write ${\cal L}(X_1)={\cal L}(X_1,X_1)$.

Consider the boundary value problem
\[
\mbox{div $w$}=f \;\;\mbox{in $G$}, \qquad
w|_{\partial G}=0,
\]
where $G$ is a bounded domain in $\mathbb R^3$ with Lipschitz boundary
$\partial G$.
Let $1<q<\infty$.
Given $f\in L^q(G)$ with compatibility condition
$\int_G f\,dx=0$,
there are a lot of solutions, some of which were found by many authors,
see Galdi \cite[Notes for Chapter III]{Ga-b}.
Among them a particular solution discovered by Bogovskii \cite{B}
is useful to recover the solenoidal condition in a cut-off procedure
on account of some fine properties of his solution.
The operator $f\mapsto \mbox{his solution $w$}$,
called the Bogovskii operator, is well defined as follows
(for details, see \cite{BS}, \cite{Ga-b}):
there is a linear operator
$\mathbb B_G: C_0^\infty(G)\to C_0^\infty(G)^3$ such that,
for $1<q<\infty$ and $k\geq 0$ integers,
\begin{equation}
\|\nabla^{k+1}\mathbb B_Gf\|_{q,G}\leq C\|\nabla^kf\|_{q,G}
\label{bog-est-1}
\end{equation}
with some $C=C(G,q,k)>0$,
which is invariant with respect to dilation of the domain $G$, and that
\begin{equation}
\mbox{div $(\mathbb B_Gf)$}=f \qquad
\mbox{if}\;\; \int_G f(x)\,dx=0.
\label{bogov}
\end{equation}
By continuity, $\mathbb B_G$ extends uniquely to a bounded operator
from $W^{k,q}_0(G)$ to $W^{k+1,q}_0(G)^3$.
In \cite[Theorem 2.5]{GHH-b} Geissert, Heck and Hieber proved that
$\mathbb B_G$ can also extend to a bounded operator
from $W^{1,q^\prime}(G)^*$ to $L^q(G)^3$, that is,
\begin{equation}
\|\mathbb B_Gf\|_{q,G}\leq C\|f\|_{W^{1,q^\prime}(G)^*},
\label{bog-est-2}
\end{equation}
where $1/q^\prime+1/q=1$.
Note that this is not true from $W^{-1,q}(G)$
to $L^q(G)^3$, see Galdi \cite[Chapter III]{Ga-b}, who
nevertheless proved that
\begin{equation}
\|\mathbb B_G[\mbox{div $F$}]\|_{q,G}
\leq C\|F\|_{q,G}
\label{bog-est-3}
\end{equation}
holds true for $F\in L^q(G)^3$ satisfying the vanishing normal trace condition
$\nu\cdot F|_{\partial G}=0$ as well as
$\mbox{div $F$}\in L^q(G)$ (\cite[Theorem III.3.4]{Ga-b}).
Instead of \eqref{bog-est-2}, one can employ \eqref{bog-est-3} 
to discuss some delicate terms arising from cut-off procedures.

Let us introduce the solenoidal function space.
Let $G\subset\mathbb R^3$ be one of the following domains;
the exterior domain $D$ under consideration,
a bounded domain with $C^{1,1}$-boundary $\partial G$ and
the whole space $\mathbb R^3$.
The class $C_{0,\sigma}^\infty(G)$ consists of
all divergence-free vector fields being in $C_0^\infty(G)$.
Let $1<q<\infty$.
By $L^q_\sigma(G)$ we denote the completion of $C_{0,\sigma}^\infty(G)$
in $L^q(G)$, then it is characterized as
\[
L^q_\sigma(G)=\{u\in L^q(G);\,\mbox{div $u$}=0,\, \nu\cdot u|_{\partial G}=0\},
\]
where $\nu$ stands for the outer unit normal to $\partial G$
and $\nu\cdot u$ is understood in the sense of normal trace on $\partial G$
(this boundary condition is absent when $G=\mathbb R^3$).
The space of $L^q$-vector fields admits the Helmholtz decomposition
\[
L^q(G)=L^q_\sigma(G)\oplus
\{\nabla p\in L^q(G);\, p\in L^q_{loc}(\overline{G})\},
\]
which was proved by Fujiwara and Morimoto \cite{FM},
Miyakawa \cite{Mi} and Simader and Sohr \cite{SiS}.
By $P_G=P_{G,q}: L^q(G)\to L^q_\sigma(G)$,
we denote the Fujita-Kato projection
associated with the decompostion above.
We then see that
$P_G\in {\cal L}(W^{1,q}(G))$ as well as
$P_G\in {\cal L}(L^q(G))$.
Note the duality relation $(P_{G,q})^*=P_{G,q^\prime}$ as well as
$L^q_\sigma(G)^*=L^{q^\prime}_\sigma(G)$,
where $1/q^\prime+1/q=1$.
We simply write $P=P_D$ for the exterior domain $D$ under consideration.
Finally, we denote several positive constants by $C$, which may
change from line to line.

We are in a position to introduce the generators which are related to
\eqref{IVP} and to the backward problem for the adjoint system subject
to the final condition at $t>0$:
\begin{equation}
\begin{split}
-\partial_sv
&=\Delta v-(\eta(s)+\omega(s)\times y)\cdot\nabla v
+\omega(s)\times v+\nabla\sigma, \\
\mbox{div $v$}&=0, \\
v|_{\partial D}&=0, \\ 
v&\to 0\quad\mbox{as $|y|\to\infty$}, \\
v(\cdot,t)&=g,
\end{split}
\label{back}
\end{equation}
in $D\times [0,t)$, where $\{v(y,s),\sigma(y,s)\}$ is the pair of unknowns.
Let us define the operators $L_\pm(t)$ by
\begin{equation}
\begin{split}
D_q(L_\pm(t))&=\{u\in L^q_\sigma(D)\cap W^{1,q}_0(D)\cap W^{2,q}(D);\,
(\omega(t)\times x)\cdot\nabla u\in L^q(D)\}, \\
L_\pm(t)u&=-P[\Delta u \pm(\eta(t)+\omega(t)\times x)\cdot\nabla u
\mp\omega(t)\times u]   \\
\end{split}
\label{generator}
\end{equation}
Then we have
\begin{equation}
\langle L_\pm(t)u, v\rangle_D=\langle u, L_\mp(t)v\rangle_D
\label{skew}
\end{equation}
for all $u\in D_q(L_\pm(t))$ and $v\in D_{q^\prime}(L_\mp(t))$,
see \cite[(2.12)]{Hi18}, where $1/q^\prime+1/q=1$.
Since the domain is time-dependent, as in Hansel and Rhandi \cite{HR14},
we need the regularity spaces
\begin{equation}
\begin{split}
Y_q(D)&=\{u\in L^q_\sigma(D)\cap W^{1,q}_0(D)\cap W^{2,q}(D);\,
|x|\nabla u\in L^q(D)\}, \\
Z_q(D)&=\{u\in L^q_\sigma(D)\cap W^{1,q}(D);\, |x|\nabla u\in L^q(D)\},
\end{split}
\label{auxi}
\end{equation}
which are Banach spaces endowed with norms
\[
\|u\|_{Y_q(D)}=\|u\|_{W^{2,q}(D)}+\||x|\nabla u\|_q,  \quad
\|u\|_{Z_q(D)}=\|u\|_{W^{1,q}(D)}+\||x|\nabla u\|_q,
\]
respectively.
Note that $Y_q(D)\subset D_q(L_\pm(t))$ for every $t\geq 0$ and that,
differently from \cite{HR14},
the homogeneous Dirichlet condition at $\partial D$ is not
involved in the space $Z_q(D)$.
The reason why this modification is actually needed will be clarified
in Section \ref{regularity}.

Hansel and Rhandi \cite{HR14} proved the following.
\begin{proposition}
[\cite{HR14}]
Suppose that $\eta$ and $\omega$ fulfill \eqref{loc-hoelder}
for some $\theta\in (0,1)$.
Let $1<q<\infty$.
The operator family $\{L_+(t)\}_{t\geq 0}$ 
generates an evolution operator $\{T(t,s)\}_{t\geq s\geq 0}$
on $L^q_\sigma(D)$ such that $T(t,s)$ is a bounded operator from 
$L^q_\sigma(D)$ into itself with
the semigroup property
\begin{equation}
T(t,\tau)T(\tau,s)=T(t,s) \quad (t\geq\tau\geq s\geq 0); \qquad T(s,s)=I,
\label{semi}
\end{equation}
in ${\cal L}(L^q_\sigma(D))$ and that the map
\[
\{t\geq s\geq 0\}\ni (t,s)\mapsto T(t,s)f\in L^q_\sigma(D)
\]
is continuous for every $f\in L^q_\sigma(D)$.
Furthermore, we have the following properties.

\begin{enumerate}
\item
Let $q\leq r<\infty$.
For each ${\cal T}\in (0,\infty)$ and $m\in (0,\infty)$, there is a constant
$C=C({\cal T},m,q,r,\theta,D)>0$ such that \eqref{LqLr} and \eqref{LqLr-grad}
hold for all $(t,s)$
with $0\leq s<t\leq{\cal T}$ and $f\in L^q_\sigma(D)$ whenever
\[
\sup_{0\leq t\leq{\cal T}}(|\eta(t)|+|\omega(t)|)
+\sup_{0\leq s<t\leq{\cal T}}
\frac{|\eta(t)-\eta(s)|+|\omega(t)-\omega(s)|}{(t-s)^\theta}\leq m.
\]

\item
Let $3/2<q<\infty$ and fix $s\geq 0$.
For every $f\in Z_q(D)$ and $t\in (s,\infty)$, we have
$T(t,s)f\in Y_q(D)$ and
\[
T(\cdot,s)f\in C^1((s,\infty); L^q_\sigma(D))
\]
with
\begin{equation}
\partial_tT(t,s)f+L_+(t)T(t,s)f=0, \qquad t\in (s,\infty),
\label{evo-eqn}
\end{equation}
in $L^q_\sigma(D)$.

\item
Fix $t>0$.
For every $f\in Y_q(D)$, we have
\[
T(t,\cdot)f\in C^1([0,t]; L^q_\sigma(D))
\]
with
\[
\partial_sT(t,s)f=T(t,s)L_+(s)f, \qquad  s\in [0,t],
\]
in $L^q_\sigma(D)$.
\end{enumerate}
\label{basic}
\end{proposition}

Among the assertions above, the second one
tells us that $T(t,s)f$ provides a strong solution without assuming
$f|_{\partial D}=0$ nor $\nabla^2f\in L^q(D)$.
This is a slight improvement of the corresponding result in 
\cite[Theorem 2.4 (b)]{HR14},
which claims the same for $f\in Y_q(D)$.
The proof of this improvement only in the second assertion
will be given in Section \ref{regularity}.
The restriction $q\in (3/2,\infty)$ stems from Lemma \ref{proj}
(and it seemed to be overlooked in \cite{HR14}).
Thus the corresponding part of Proposition 2.1 of \cite{Hi18}
should be replaced by the second assertion above.
Nevertheless, we observe that
the semigroup property \eqref{semi} in ${\cal L}(L^q_\sigma(D))$
holds still for every $q\in (1,\infty)$.
In fact, given $f\in C_{0,\sigma}^\infty(D)$,
it follows from the second and third assertions that
$\partial_\tau\big(T(t,\tau)T(\tau,s)f\big)=0$
in $L^q_\sigma(D)$ with $q\in (3/2,\infty)$, yielding
$T(t,\tau)T(\tau,s)f=T(t,s)f$.
Once we have that for all $f\in C_{0,\sigma}^\infty(D)$,
a continuity argument leads to the same equality for all
$f\in L^q_\sigma(D)$ with $q\in (1,\infty)$.

We should mention that the results obtained in the previous study
\cite{Hi18} are still valid in spite of the restriction $q\in (3/2,\infty)$
above.
Let $S(t,s)$ be the evolution operator generated by the backward problem
\begin{equation}
-\partial_sv(s)+L_-(s)v(s)=0, \quad s\in [0,t); \qquad
v(t)=g
\label{backward}
\end{equation}
in $L^q_\sigma(D)$, which corresponds to \eqref{back}.
It is given by
\begin{equation}
S(t,s)=\widetilde T(t-s,0; t), \qquad t\geq s\geq 0,
\label{back-op}
\end{equation}
where $\{\widetilde T(\tau,s; t)\}_{0\leq s\leq \tau\leq t}$ 
is the evolution operator generated by
the related initial value problem
\begin{equation}
\partial_\tau w(\tau)+L_-(t-\tau)w(\tau)=0, \quad \tau\in (s,t]; \qquad
w(s)=g,
\label{ivp-adj}
\end{equation}
see \cite[Subsection 2.3]{Hi18}.
For \eqref{ivp-adj}, note that $t>0$ is just a parameter 
appearing in the coefficient of the equation.
We then have the duality relation \cite[Lemma 2.1]{Hi18}
\begin{equation}
T(t,s)^*=S(t,s), \qquad S(t,s)^*=T(t,s) \qquad
\mbox{in ${\cal L}(L^q_\sigma(D))$}
\label{duality}
\end{equation}
for $t\geq s\geq 0$,
which plays an important role in \cite{Hi18}.
In fact, given $f,\, g\in C_{0,\sigma}^\infty(D)$
(instead of $f\in Y_{q^\prime}(D),\, g\in Y_q(D)$ in the proof
of \cite[Lemma 2.1]{Hi18}, where $1/q^\prime+1/q=1$), we obtain
\begin{equation}
\langle T(t,s)f, g\rangle_D
=\langle f, S(t,s)g\rangle_D
\label{dual-TS}
\end{equation}
by computing
$\partial_\tau \langle T(\tau,s)f, S(t,\tau)g\rangle_D=0$ with use of
\eqref{evo-eqn} as well as
\begin{equation}
-\partial_s S(t,s)g+L_-(s)S(t,s)g=0, \qquad s\in [0,t),
\label{adj-evo}
\end{equation}
in $L^q_\sigma(D)$, where $\langle\cdot, \cdot\rangle_D$
should be understood for the pair of
$L^{q^\prime}_\sigma(D)$ and $L^q_\sigma(D)$ with $q\in (3/2,3)$.
Once we have \eqref{dual-TS} for all $f,\, g\in C_{0,\sigma}^\infty(D)$,
we have only to perform a continuity argument to justify
\eqref{duality} for every $q\in (1,\infty)$.
In addition, as emphasized in Section \ref{intro},
one of key ingredients in \cite{Hi18} is the energy relation
which we certainly have since the second assertion of Proposition \ref{basic}
is available in $L^2_\sigma(D)$.
Finally, as described in \cite[Section 4]{Hi18} for the proof
of decay estimates,
it suffices to carry out a cut-off procedure for fine
initial velocities being in $C_{0,\sigma}^\infty(D)$, so that
the restriction $q\in (3/2,\infty)$ does not cause any problem.

We recall the following $L^q$-$L^r$ estimates globally in time
developed by the present author
\cite[Theorem 2.1, Proposition 3.1]{Hi18}.
Let us introduce
\begin{equation}
\begin{split}
&|(\eta,\omega)|_0=\sup_{t\geq 0}\;(|\eta(t)|+|\omega(t)|), \\
&|(\eta,\omega)|_\theta=\sup_{t>s\geq 0}
\frac{|\eta(t)-\eta(s)|+|\omega(t)-\omega(s)|}{(t-s)^\theta}
\end{split}
\label{hoelder}
\end{equation}
and
\begin{equation}
\Lambda(\tau_*)
=\{(t,s);\,t>s\geq 0,\; t-s\leq\tau_*\}
\label{para-set}
\end{equation}
for $\tau_*\in (0,\infty)$.
\begin{proposition}
[\cite{Hi18}]
Suppose that $\eta$ and $\omega$ fulfill \eqref{rigid} for 
some $\theta\in (0,1)$.
Let $1<q\leq r<\infty$.

\begin{enumerate}
\item
For each $m\in (0,\infty)$, there is a constant $C=C(m,q,r,\theta,D)>0$ 
such that
\begin{equation}
\begin{split}
&\|T(t,s)f\|_r\leq C(t-s)^{-(3/q-3/r)/2}\|f\|_q, \\
&\|T(t,s)^*g\|_r\leq C(t-s)^{-(3/q-3/r)/2}\|g\|_q,
\end{split}
\label{LqLr-noch}
\end{equation}
for all
$t>s\geq 0$ and $f,\, g\in L^q_\sigma(D)$ whenever
\begin{equation}
|(\eta,\omega)|_0+|(\eta,\omega)|_\theta\leq m
\label{rigid-bound}
\end{equation}
is satisfied.

\item
Given $\tau_*\in (0,\infty)$ and $m\in (0,\infty)$,
let $\Lambda(\tau_*)$ be as in \eqref{para-set} and assume
\eqref{rigid-bound}.
Then there is a constant
$C=C(\tau_*,m,q,r,\theta,D)>0$
such that
\begin{equation}
\begin{split}
&\|\nabla T(t,s)f\|_r\leq C(t-s)^{-(3/q-3/r)/2-1/2}\|f\|_q,  \\
&\|\nabla T(t,s)^*g\|_r\leq C(t-s)^{-(3/q-3/r)/2-1/2}\|g\|_q,
\end{split}
\label{LqLr-grad-noch}
\end{equation}
for all $(t,s)\in\Lambda(\tau_*)$
and $f,\, g\in L^q_\sigma(D)$.

\end{enumerate}
\label{ann-18}
\end{proposition}

The point of the second assertion is that the constant $C>0$ 
in \eqref{LqLr-grad-noch} can be taken uniformly in $(t,s)$
with $t-s\leq\tau_*$.
This must be the first step toward \eqref{LqLr-grad-noch} for all $t>s\geq 0$.
It was not covered by \cite{HR14} but shown by \cite[Proposition 3.1]{Hi18} 
under the condition \eqref{rigid},
however, only for $\nabla T(t,s)$.
The same result for $\nabla T(t,s)^*$ follows from the one for
$\nabla\widetilde T(\tau,s; t)$,
which is the solution operator to \eqref{ivp-adj} and can be
constructed along the procedure adopted by \cite{HR14},
see also Section \ref{regularity} of this paper.
To this end, as clarified in \cite[Subsections 3.1--3.3]{Hi18},
it suffices to investigate the initial value problem
for the same equation as in \eqref{ivp-adj} over a bounded domain $D_R$ 
with $R>0$ large enough by following the Tanabe-Sobolevskii theory \cite{T}.
Taking a look at
the generator $L_-(t-\tau)$ together with the condition \eqref{rigid},
we observe that
all the constants in several key estimates can be taken uniformly in 
$(\tau,s)$ with $\tau-s\leq\tau_*$,
see the proof of Lemma 3.2 of \cite{Hi18}, which implies
\[
\|\nabla \widetilde T(\tau,s; t)g\|_r\leq C(\tau-s)^{-(3/q-3/r)/2-1/2}\|g\|_q
\]
for all $(\tau,s)$ with $\tau-s\leq\tau_*$ as well as $0\leq s<\tau\leq t$
and $1<q\leq r<\infty$,
where $C>0$ depends on $\tau_*\in (0,t)$ but is independent of $t>0$.
By \eqref{back-op} and \eqref{duality}
we conclude that $\nabla T(t,s)^*$ also satisfies \eqref{LqLr-grad-noch}
for all $(t,s)\in\Lambda(\tau_*)$.

We are now in a position to present the main result of this paper.
\begin{theorem}
Suppose that $\eta$ and $\omega$ fulfill \eqref{rigid} for some 
$\theta\in (0,1)$.

\begin{enumerate}
\item
Let $1<q\leq r\leq 3$.
For each $m\in (0,\infty)$, there is a constant $C=C(m,q,r,\theta,D)>0$ 
such that \eqref{LqLr-grad-noch} holds for all
$t>s\geq 0$ and $f,\,g\in L^q_\sigma(D)$ whenever \eqref{rigid-bound}
is satisfied.

\item
Let $1<q\leq r$ as well as $r\in (3,\infty)$.
For each $m\in (0,\infty)$, there is a constant $C=C(m,q,r,\theta,D)>0$ 
such that
\begin{equation}
\begin{split}
&\|\nabla T(t,s)f\|_r\leq C(t-s)^{-3/2q}\|f\|_q,  \\
&\|\nabla T(t,s)^*g\|_r\leq C(t-s)^{-3/2q}\|g\|_q,
\end{split}
\label{grad-modi}
\end{equation}
for all $(t,s)$ with
\[
t-s>2 \quad\mbox{as well as $0\leq s<t$}
\]
and $f,\,g\in L^q_\sigma(D)$ whenever \eqref{rigid-bound} is satisfied.

\item
Let $1<q<\infty$.
For each $m\in (0,\infty)$, there is a constant $C=C(m,q,\theta,D)>0$ 
such that \eqref{LqLr-noch} with $r=\infty$ holds true, that is,
\begin{equation}
\begin{split}
&\|T(t,s)f\|_\infty\leq C(t-s)^{-3/2q}\|f\|_q,   \\
&\|T(t,s)^*g\|_\infty\leq C(t-s)^{-3/2q}\|g\|_q,
\end{split}
\label{L-inf}
\end{equation}
for all $t>s\geq 0$ and $f,\,g\in L^q_\sigma(D)$
whenever \eqref{rigid-bound} is satisfied.
\end{enumerate}
\label{main}
\end{theorem}
\begin{remark}
Maremonti and Solonnikov \cite{MSol} first pointed out that the restriction
$1<q\leq r\leq 3=n \;(\mbox{space dimension})$
for the desired rate \eqref{LqLr-grad-noch} of decay
is optimal when $\eta=\omega=0$.
Later on, in this case of the Stokes semigroup, the present author \cite{Hi11}
gave another proof of the optimality, where a key observation is that
the issue is closely related to summability of
the steady Stokes flow near spatial infinity.
From this point of view, it is also conjectured by 
\cite[Section 5]{Hi11} that the desired rate \eqref{LqLr-grad} of decay
could be obtained for $1<q\leq r\leq 6=n(n+1)/(n-1)$
when the translation of the body is present, that is, $\eta\neq 0$.
For the Stokes semigroup,
the optimality of the rate \eqref{grad-modi} of decay was also proved by
Maremonti and Solonnikov \cite{MSol}
in the sense that better rate $(t-s)^{-3/2q-\varepsilon}$
with some $\varepsilon >0$
is impossible when $r>3$.
\label{optimal}
\end{remark}

Having several applications to the Navier-Stokes system in mind,
we next provide useful estimates especially for the adjoint evolution operator.
Let us introduce the Lorentz spaces which are usually defined
as Banach spaces in terms of the average function of the rearrangement,
see \cite{BL} for details.
For simplicity, we just define the solenoidal Lorentz spaces by
\[
L^{q,\rho}_\sigma(D)=
\left(L^{q_0}_\sigma(D), L^{q_1}_\sigma(D) \right)_{\theta,\rho}
\]
with
\[
1<q_0<q<q_1<\infty, \quad
\frac{1}{q}=\frac{1-\theta}{q_0}+\frac{\theta}{q_1}, \quad
1\leq\rho\leq\infty,
\]
where $(\cdot,\cdot)_{\theta,\rho}$ denotes the real interpolation functor.
Then the RHS above is independent of choice of $\{q_0,q_1\}$, so that the
space $L^{q,\rho}_\sigma(D)$, whose norm is denoted by $\|\cdot\|_{q,\rho}$,
is well-defined.
It is obvious by interpolation to obtain
\eqref{LqLr-noch} and \eqref{LqLr-grad-noch}
for all $t>s\geq 0$ in which the Lebesgue spaces are replaced by
the Lorentz spaces
except for \eqref{LqLr-grad-noch} with $1<q\leq r=3$.
But we do need this end-point case for the adjoint
evolution operator to study the large time behavior of
the Navier-Stokes flow around a background flow
(such as steady flow and time-periodic one)
that decays with scale-critical rate at spatial infinity, 
see \cite{HiS}, \cite{Y}.
For completeness,
it is worse while providing \eqref{adj-Lorentz} below 
including the nontrivial case $r=3$.
Once we have \eqref{adj-Lorentz}, we can get \eqref{adj-int} by following
the argument developed by Yamazaki \cite{Y}.
\begin{theorem}
Let $1<q\leq r\leq 3$ and $1\leq\rho <\infty$.
Let $m\in (0,\infty)$ and assume \eqref{rigid-bound}.
Then there is a constant
$C=C(m,q,r,\rho,\theta,D)>0$ such that
\begin{equation}
\|\nabla T(t,s)^*g\|_{r,\rho}
\leq C(t-s)^{-(3/q-3/r)/2-1/2}\|g\|_{q,\rho}
\label{adj-Lorentz}
\end{equation}
for all $t>s\geq 0$ and $g\in L^{q,\rho}_\sigma(D)$.
If in particular $1/q-1/r=1/3$ as well as $1<q<r\leq 3$,
then there is a constant $C=C(m,q,\theta,D)>0$ such that
\begin{equation}
\int_0^t\|\nabla T(t,s)^*g\|_{r,1}\,ds
\leq C\|g\|_{q,1}
\label{adj-int}
\end{equation}
for all $t>0$ and $g\in L^{q,1}_\sigma(D)$.
\label{adj}
\end{theorem}

\section{Whole space problem}
\label{whole}

In this section we consider the non-autonomous system
\begin{equation}
\begin{split}
&\partial_tu
=\Delta u+(\eta(t)+\omega(t)\times x)\cdot\nabla u
-\omega(t)\times u-\nabla p,  \\
&\mbox{div $u$}=0
\end{split}
\label{eqn-wh}
\end{equation}
in $\mathbb R^3\times (s,\infty)$ subject to
\begin{equation}
\begin{split}
&u\to 0\quad\mbox{as $|x|\to\infty$}, \\
&u(\cdot,s)=f.
\end{split}
\label{side-wh}
\end{equation}
Indeed the system was studied by \cite{CMi}, \cite{GHa}, \cite{Ha},
\cite{HR11} and \cite{HR14}, but
we have to supplement a couple of regularity properties:
Lemma \ref{whole-basic} on some smoothing actions
and Lemma \ref{weak-whole} on the time derivative for general 
$f\in L^q_\sigma(\mathbb R^3)$.

As long as $f$ fulfills the compatibility condition
$\mbox{div $f$}=0$,
we see that $\nabla p=0$ within the class $\nabla p\in L^q(\mathbb R^3)$ and 
that the solution is just the heat semigroup in which a change of variables
is made in an appropriate way, because
\begin{equation}
\mbox{div $[(\eta+\omega\times x)\cdot\nabla u-\omega\times u]$}
=(\eta+\omega\times x)\cdot\nabla\mbox{div $u$}=0.
\label{dri-sole}
\end{equation}
In fact, 
the solution to \eqref{eqn-wh}--\eqref{side-wh} is explicitly described as
\begin{equation}
\begin{split}
u(x,t)
&=\big(U(t,s)f\big)(x) \\
&=\Phi(t,s)\left(e^{(t-s)\Delta} f\right)\left(
\Phi(t,s)^\top\left(x+\int_s^t\Phi(t,\tau)\eta(\tau)\,d\tau\right)\right)  \\
\end{split}
\label{evo-wh}
\end{equation}
where
\[
\big(e^{t\Delta}f\big)(x)=(4\pi t)^{-3/2}\left(e^{-|\cdot|^2/4t}*f\right)(x),
\]
while $3\times 3$ orthogonal matrix $\Phi(t,s)$ stands for
the evolution operator for the ordinary differential equation
$\frac{d}{dt}\varphi=-\omega\times\varphi$, 
see the literature above for details.
By $\Gamma(x,y;t,s)$ we denote the fundamental solution, thst is,
the kernal matrix of \eqref{evo-wh}:
\[
u(x,t)=\int_{\mathbb R^3}\Gamma(x,y;t,s)f(y)\,dy.
\]
Then the adjoint of $U(t,s)$ is given by
\begin{equation}
\big(U(t,s)^*g\big)(y)
=\int_{\mathbb R^3}\Gamma(x,y;t,s)^\top g(x)\,dx.
\label{backevo-wh}
\end{equation}
Given $t>0$ (final time) and a suitable solenoidal vector field
$g$ (final data), the velocity $v(s)=U(t,s)^*g$ together with
the trivial pressure gradient $\nabla\sigma=0$ formally
(even rigorously for fine $g$, see \cite[third assertion of Lemma 3.1]{Hi18})
solves the backward system
\begin{equation}
\begin{split}
&-\partial_sv
=\Delta v-(\eta(s)+\omega(s)\times y)\cdot\nabla v
+\omega(s)\times v+\nabla\sigma,  \\
&\mbox{div $v$}=0,
\end{split}
\label{back-wh}
\end{equation}
in $\mathbb R^3\times [0,t)$ subject to
\begin{equation}
\begin{split}
&v\to 0\quad\mbox{as $|y|\to\infty$}, \\
&v(\cdot,t)=g.
\end{split}
\label{backside-wh}
\end{equation}
The initial value problem corresponding to \eqref{ivp-adj} is given  by
\begin{equation}
\begin{split}
\partial_\tau w
&=\Delta w-(\eta(t-\tau)+\omega(t-\tau)\times y)\cdot\nabla w
+\omega(t-\tau)\times w+\nabla p_w, \\
\mbox{div $w$}&=0, \\
w&\to 0 \quad\mbox{as $|y|\to\infty$}, \\
w(\cdot,s)&=g,
\end{split}
\label{ivp-adj-wh}
\end{equation}
in $\mathbb R^3\times (s,t]$
(with $\nabla p_w=0$ under the compatibility condition $\mbox{div $g$}=0$),
where $t>0$ is just a parameter.
The solution to \eqref{ivp-adj-wh} is described as
\begin{equation}
\begin{split}
w(y,\tau)
&=\big(\widetilde U(\tau,s; t)g\big)(y)  \\
&=\Phi(t-\tau,t-s)\big(e^{(\tau-s)\Delta}g\big)(\cdots)
\end{split}
\label{evo-adj-wh}
\end{equation}
with
\[
(\cdots)=
\Phi(t-\tau,t-s)^\top
\left(y-\int_s^\tau\Phi(t-\tau,t-\sigma)\eta(t-\sigma)\,d\sigma\right)
\]
where the orthogonal matrix 
$\Phi(\cdot,\cdot)$ is the same as in \eqref{evo-wh}.
It is verified that the relation
\[
U(t,s)^*=\widetilde U(t-s,0; t), \qquad t\geq s\geq 0,
\]
recovers \eqref{backevo-wh} as in \eqref{back-op}.

Although we will provide the results 
(Lemma \ref{whole-basic}, Lemma \ref{weak-whole})
only on the evolution operator $U(t,s)$, 
those for the adjoint $U(t,s)^*$ or $\widetilde U(\tau,s; t)$
are also available and will be needed
to obtain the assertions for the adjoint $T(t,s)^*$.

Let $1<q<\infty$.
Correspondingly to the auxilliary spaces 
\eqref{auxi} for the exterior problem, let us introduce
\begin{equation*}
\begin{split}
&Z_q(\mathbb R^3)
=\{u\in L^q_\sigma(\mathbb R^3)\cap W^{1,q}(\mathbb R^3);
|x|\nabla u\in L^q(\mathbb R^3)\}, \\
&Y_q(\mathbb R^3)
=Z_q(\mathbb R^3)\cap W^{2,q}(\mathbb R^3),
\end{split}
\end{equation*}
to describe the regularity of the solution.
We note that, under the condition \eqref{loc-hoelder} solely,
the regularity deduced in the following lemma holds true 
subject to estimates \eqref{est-YZ-whole}--\eqref{est-Z-whole} below
for $0\leq s<t\leq {\cal T}$ with $C>0$ 
that depends on ${\cal T}\in (0,\infty)$.
Nevertheless, for later use, 
we will show those estimates for $(t,s)\in\Lambda(\tau_*)$,
see \eqref{para-set},
under the additional assumption $\eta\in L^\infty(0,\infty; \mathbb R^3)$ 
(even under \eqref{rigid}).
\begin{lemma}
Suppose that $\eta$ and $\omega$ fulfill \eqref{loc-hoelder}
for some $\theta\in (0,1)$.
Assume in addition that
$\eta\in L^\infty(0,\infty; \mathbb R^3)$ for the second, third and fourth
assertions below.
Let $1<q<\infty$.
Then $\{U(t,s)\}_{t\geq s\geq 0}$ given by \eqref{evo-wh} 
defines an evolution operator on $L^q(\mathbb R^3)$ and on
$L^q_\sigma(\mathbb R^3)$.
Furthermore, we have the following properties.

\begin{enumerate}
\item
Let $q\leq r\leq\infty$.
For every integer $j\geq 0$, there is a constant $c_j=c_j(q,r)>0$, 
independent of $\eta$ and $\omega$, such that
\begin{equation}
\begin{split}
&\nabla^jU(\cdot,s)f\in C((s,\infty); L^r(\mathbb R^3)), \\
&\|\nabla^jU(t,s)f\|_{r,\mathbb R^3}
\leq c_j(t-s)^{-(3/q-3/r)/2-j/2}\|f\|_{q,\mathbb R^3}
\end{split}
\label{LqLr-whole}
\end{equation}
for all $t>s\geq 0$ and $f\in L^q(\mathbb R^3)$.

\item
Let $q\leq r<\infty$ and $m\in (0,\infty)$.
For every $f\in Z_q(\mathbb R^3)$ and $t\in (s,\infty)$,
we have $|x|\nabla U(t,s)f\in L^r(\mathbb R^3)$ subject to
\begin{equation}
\begin{split}
&\quad \||x|\nabla U(t,s)f\|_{r,\mathbb R^3} \\ 
&\leq C(t-s)^{-(3/q-3/r)/2}\||x|\nabla f\|_{q,\mathbb R^3} \\
&\quad +C(t-s)^{-(3/q-3/r)/2+1/2}\{1+m(t-s)^{1/2}\}\|\nabla f\|_{q,\mathbb R^3}
\end{split}
\label{drift-whole}
\end{equation}
for all $t>s\geq 0$ with some constant
$C=C(q,r)>0$, whenever $|\eta|_0:=\sup_{t\geq 0}|\eta(t)|\leq m$.
 
\item
For every $f\in Z_q(\mathbb R^3)$ and $t\in (s,\infty)$, we have 
$U(t,s)f\in Y_q(\mathbb R^3)$ and
\[
u:=U(\cdot,s)f\in C^1((s,\infty); L^q_\sigma(\mathbb R^3))
\]
with \eqref{eqn-wh}--\eqref{side-wh} in $L^q_\sigma(\mathbb R^3)$.
Let $\tau_*\in (0,\infty)$ and $m\in (0,\infty)$.
If in addition \eqref{rigid} is assumed, then
there is a constant $C=C(\tau_*,m,q)>0$ such that
\begin{equation}
\|U(t,s)f\|_{Y_q(\mathbb R^3)}+\|\partial_tU(t,s)f\|_{q,\mathbb R^3}
\leq C(t-s)^{-1/2}\|f\|_{Z_q(\mathbb R^3)}
\label{est-YZ-whole}
\end{equation}
for all $(t,s)\in\Lambda(\tau_*)$ and $f\in Z_q(\mathbb R^3)$
whenever \eqref{rigid-bound} is satisfied,
where $\Lambda(\tau_*)$ is given by \eqref{para-set}.

\item
Let $q\leq r<\infty$, $\tau_*\in (0,\infty)$ and $m\in (0,\infty)$.
For every $f\in Z_q(\mathbb R^3)$ and $t\in (s,\infty)$, 
we have $U(t,s)f\in Z_r(\mathbb R^3)$ subject to
\begin{equation}
\|U(t,s)f\|_{Z_r(\mathbb R^3)}
\leq C(t-s)^{-(3/q-3/r)/2}\|f\|_{Z_q(\mathbb R^3)}
\label{est-Z-whole}
\end{equation}
for all $(t,s)\in \Lambda(\tau_*)$ 
with some constant $C=C(\tau_*,m,q,r)>0$
whenever $|\eta|_0\leq m$.
\end{enumerate}
\label{whole-basic}
\end{lemma}

\begin{proof}
The first assertion follows from the corresponding properties 
of the heat semigroup.
The third assertion is a slight improvement of the one 
in \cite{HR11} and \cite{HR14}, but it follows from knowledge
obtained there (see Proposition 3.1 (a) of \cite{HR14}).
The second and fourth assertions for the case $r>q$ are new 
and preparations for Lemma \ref{smoothing-Z}.

As in the proof of \eqref{drift-whole} with $r=q$ by \cite{HR11},
we have
\begin{equation*}
\begin{split}
&\quad |x||\nabla (U(t,s)f)(x)|  \\
&\leq\int_{\mathbb R^3}
(|x-y|+|y|)
\frac{e^{-|x-y|^2/4(t-s)}}{\{4\pi(t-s)\}^{3/2}}
\left|(\nabla f)\left(\Phi(t,s)^\top(y+h_{t,s})\right)\right|\,dy  \\
&=:I+J,
\end{split}
\end{equation*}
where
$h_{t,s}:=\int_s^t \Phi(t,\tau)\eta(\tau)\,d\tau$.
We then find that
\[
\|I\|_{r,\mathbb R^3}\leq C(t-s)^{-(3/q-3/r)/2+1/2}\|\nabla f\|_{q,\mathbb R^3}
\]
and that
\begin{equation*}
\begin{split}
\|J\|_{r,\mathbb R^3}
&\leq C(t-s)^{-(3/q-3/r)/2}\left\|
|\cdot|(\nabla f)\left(\Phi(t,s)^\top(\,\cdot+h_{t,s})\right)
\right\|_{q,\mathbb R^3}  \\
&\leq C(t-s)^{-(3/q-3/r)/2}\big\{
\||\cdot|\nabla f\|_{q,\mathbb R^3}+|\eta|_0(t-s)\|\nabla f\|_{q,\mathbb R^3}
\big\}.
\end{split}
\end{equation*}
They thus imply \eqref{drift-whole}.
It is easily seen that
\[
\|\nabla^{j+1}U(t,s)f\|_{r,\mathbb R^3}
\leq C(t-s)^{-(3/q-3/r)/2-j/2}\|\nabla f\|_{q,\mathbb R^3}
\]
for all $t>s\geq 0$, $1<q\leq r<\infty$ and $j=0,1$, 
which together with \eqref{LqLr-whole}--\eqref{drift-whole}
(and by using the equation \eqref{eqn-wh} for $\partial_tU(t,s)f$)
leads to \eqref{est-YZ-whole} as well as \eqref{est-Z-whole}.
The proof is complete.
\end{proof}

It is natural to expect that $U(t,s)f$ is a weak solution in a sense 
together with a reasonable estimate of $\partial_tU(t,s)f$
even if $f\in L^q_\sigma(\mathbb R^3)$
rather than $f\in Z_q(\mathbb R^3)$.
The following lemma gives an affirmative answer.
Indeed the assumption \eqref{loc-hoelder} is enough to obtain the assertion,
but the constant in 
\eqref{weak-est-whole} below depends on ${\cal T}\in (0,\infty)$ 
for $0\leq s<t\leq{\cal T}$.
For later use, 
it is convenient to show the following form when assuming \eqref{rigid}.
\begin{lemma}
Suppose that $\eta$ and $\omega$ fulfill \eqref{rigid} for some 
$\theta\in (0,1)$.
Let $1<q<\infty$ and $R>0$.
Given $f\in L^q_\sigma(\mathbb R^3)$ and $s\geq 0$, we set $u(t)=U(t,s)f$.
For each $\tau_*\in (0,\infty)$ and $m\in (0,\infty)$, there is a constant
$C=C(\tau_*,m,q,R)>0$ such that
\begin{equation}
u\in C^1((s,\infty); W^{-1,q}(B_R)),
\label{weak-deri-wh}
\end{equation}
\begin{equation}
\|\partial_tU(t,s)f\|_{W^{-1,q}(B_R)}\leq C(t-s)^{-1/2}\|f\|_{q,\mathbb R^3}
\label{weak-est-whole}
\end{equation}
for all $(t,s)\in\Lambda(\tau_*)$
and $f\in L^q_\sigma(\mathbb R^3)$ whenever \eqref{rigid-bound} is satisfied,
where $\Lambda(\tau_*)$ is given by \eqref{para-set}.
Furthermore, we have
\begin{equation}
\langle \partial_t u,\psi\rangle_{B_R}
+\langle\nabla u+u\otimes (\eta+\omega\times x)
-(\omega\times x)\otimes u, \nabla\psi\rangle_{B_R}=0
\label{weak-eqn-wh}
\end{equation}
for all $t\in (s,\infty)$ and $\psi\in W_0^{1,q^\prime}(B_R)^3$, where
$1/q^\prime+1/q=1$.
\label{weak-whole}
\end{lemma}

\begin{proof}
Given $f\in C^\infty_{0,\sigma}(\mathbb R^3)$ and $s\geq 0$, 
we set $u(t)=U(t,s)f$, which satisfies \eqref{weak-eqn-wh} for
every $\psi\in C^\infty_0(B_R)^3$.
From this together with \eqref{LqLr-whole} we see that
\begin{equation*}
\begin{split}
|\langle \partial_tu, \psi\rangle_{B_R}|
&\leq \big\{\|\nabla u\|_{q,\mathbb R^3}+m(1+2R)\|u\|_{q,\mathbb R^3}\big\}
\|\nabla\psi\|_{q^\prime,B_R}  \\
&\leq C\left\{1+m(1+2R)\sqrt\tau_*\right\}(t-s)^{-1/2}
\|f\|_{q,\mathbb R^3}\|\nabla\psi\|_{q^\prime,B_R}
\end{split}
\end{equation*}
as long as $t-s\leq\tau_*$.
We thus obtain \eqref{weak-est-whole} for 
$f\in C^\infty_{0,\sigma}(\mathbb R^3)$.
Given $f\in L^q_\sigma(\mathbb R^3)$, 
we take $f_j\in C^\infty_{0,\sigma}(\mathbb R^3)$ which converges to $f$
as $j\to\infty$ in the norm $\|\cdot\|_{q,\mathbb R^3}$.
Then $\partial_tU(t,s)f_j$ goes to some 
$W_R(t,s)f\in W^{-1,q}(B_R)$.
Since the convergence is uniform with respect to $t$ belonging to 
any compact interval in $(s,\infty)$,
we have
$W_R(\cdot, s)f\in C((s,\infty); W^{-1,q}(B_R))$.
From this convergence with \eqref{LqLr-whole} we observe
\[
U(t,s)f=U(s+\varepsilon,s)f+\int_{s+\varepsilon}^t W_R(\tau,s)f\,d\tau
\]
in $W^{-1,q}(B_R)$,
where $\varepsilon >0$ is arbitrary.
This implies \eqref{weak-deri-wh} and
$W_R(t,s)f$ coincides with $\partial_tU(t,s)f$ for every $R>0$.
Hence, we obtain \eqref{weak-est-whole}.
Equation \eqref{weak-eqn-wh} is easily verified
by approximation procedure above.
\end{proof}

\section{Interior problem}
\label{interior}
This section is devoted to the study of the initial value
problem for the non-autonomous system
\begin{equation}
\begin{split}
\partial_tu
&=\Delta u+(\eta(t)+\omega(t)\times x)\cdot\nabla u
-\omega(t)\times u-\nabla p, \\
\mbox{div $u$}&=0, \\
u|_{\partial D_R}&=0, \\
u(\cdot,s)&=f,
\end{split}
\label{eqn-bdd}
\end{equation}
in $D_R\times (s,\infty)$ with $R\in [R_0,\infty)$ being fixed,
where $R_0$ is as in \eqref{obstacle}.
Let $1<q<\infty$.
Let us introduce the Stokes operator
\begin{equation*}
\begin{split}
D_q(A)&=L^q_\sigma(D_R)\cap W^{1,q}_0(D_R)\cap W^{2,q}(D_R), \\
Au&=-P_{D_R}\Delta u,
\end{split}
\end{equation*}
and the operator
\begin{equation*}
\begin{split}
D_q(L_R(t))&=D_q(A), \\
L_R(t)u&=-P_{D_R}[\Delta u
+(\eta(t)+\omega(t)\times x)\cdot\nabla u-\omega(t)\times u]  \\
&=Au-(\eta(t)+\omega(t)\times x)\cdot\nabla u+\omega(t)\times u,
\end{split}
\end{equation*}
where $P_{D_R}$ denotes the Fujita-Kato projection associated
with the Helmholtz decomposition (\cite{FM}), see Section \ref{result}.
The last equality above follows from \eqref{dri-sole} and the fact 
that the normal trace of the drift term vanishes, see \cite[(3.22)]{Hi18}.

For the interior problem
one can apply the general theory of parabolic evolution operators developed by
Tanabe, see \cite[Chapter 5]{T}, to find that $\{L_R(t)\}_{t\geq 0}$ generates
an evolution operator $\{V(t,s)\}_{t\geq s\geq 0}$ on
$L^q_\sigma(D_R)$.
For every $f\in L^q_\sigma(D_R)$, we know that $u(t)=V(t,s)f$ is of class
\begin{equation}
\begin{split}
&u\in C^1((s,\infty); L^q_\sigma(D_R))
\cap C((s,\infty); D_q(A))
\cap C([s,\infty), L^q_\sigma(D_R)), \\
&\nabla p\in C((s,\infty); L^q(D_R)),
\end{split}
\label{cls-bdd}
\end{equation}
and satisfies \eqref{eqn-bdd} in $L^q_\sigma(D_R)$.
If, in addition, the pressure $p$ is chosen such that
$\int_{D_R}p\,dx=0$ for each time $t$, then
\begin{equation}
p\in C((s,\infty); L^q(D_R))
\label{cls-bdd-p}
\end{equation}
by the Poincar\'e inequality together with \eqref{cls-bdd} for $\nabla p$.

We start with the following lemma (\cite{HR14}, \cite{Hi18}).
\begin{lemma}
Suppose that $\eta$ and $\omega$ fulfill \eqref{rigid} 
for some $\theta\in (0,1)$.
Let $1<q\leq r<\infty$.
For each $\tau_*\in (0,\infty)$, $m\in (0,\infty)$ and $j=0,1$, 
there are constants
$C_j=C_j(\tau_*,m,q,r,\theta,D_R)>0$ and $C_2=C_2(\tau_*,m,q,\theta,D_R)>0$ 
such that
\begin{equation}
\|\nabla^jV(t,s)f\|_{r,D_R}\leq C_j(t-s)^{-(3/q-3/r)/2-j/2}\|f\|_{q,D_R}
\label{LqLr-bdd}
\end{equation}
\begin{equation}
\|p(t)\|_{q,D_R}\leq C_2(t-s)^{-(1+1/q)/2}\|f\|_{q,D_R}
\label{pressure-bdd}
\end{equation}
\begin{equation}
\|\partial_tV(t,s)f\|_{W^{-1,q}(D_R)}\leq C_2(t-s)^{-(1+1/q)/2}\|f\|_{q,D_R}
\label{weak-bdd}
\end{equation}
for all $(t,s)\in\Lambda(\tau_*)$ and $f\in L^q_\sigma(D_R)$ 
whenever \eqref{rigid-bound} is satisfied,
where $\Lambda(\tau_*)$ is given by \eqref{para-set}.
Here, $p(t)$ denotes the pressure associsted with $V(t,s)f$ and 
it is singled out subject to the side condition
$\int_{D_R}p\,dx=0$.
\label{basic-bdd}
\end{lemma}

\begin{proof}
$L^q$-$L^r$ estimate
\eqref{LqLr-bdd} was shown by \cite{HR14} for $0\leq s<t\leq {\cal T}$
with $C_j>0$ that depends on ${\cal T}\in (0,\infty)$
under the condition \eqref{loc-hoelder}.
The present author \cite[Lemma 3.2]{Hi18} verified that the constant $C_j$ 
can be taken uniformly in $(t,s)$ satisfying $t-s\leq\tau_*$
as long as \eqref{rigid} is fulfilled.
Set $u(t)=V(t,s)f$.
Estimate \eqref{pressure-bdd} for the pressure was also proved 
by \cite[Lemma 3.2]{Hi18} via
\begin{equation}
\|p(t)\|_{q,D_R}
\leq C\|\nabla^2u(t)\|_{q,D_R}^{1/q}\|\nabla u(t)\|_{q,D_R}^{1-1/q}
+C\|\nabla u(t)\|_{q,D_R}
\label{pre-trace}
\end{equation}
and it is a slight improvement
of the one obtained by \cite[Lemma 4.3]{HR14}.
The remarkable rate $(t-s)^{-(1+1/q)/2}$ for the pressure
near the initial time was discovered first
by \cite{HiS} for the autonomous case
(even for the Stokes system) 
and the proof relied on analysis of the resolvent.
Estimate \eqref{weak-bdd} immediately follows from
\begin{equation}
\langle \partial_tu, \psi\rangle_{D_R}
=-\langle\nabla u+u\otimes (\eta+\omega\times x)
-(\omega\times x)\otimes u, \nabla\psi\rangle_{D_R}
+\langle p, \mbox{div $\psi$}\rangle_{D_R}
\label{weak-eqn-bdd}
\end{equation}
for every $\psi\in C_0^\infty(D_R)^3$ together with 
\eqref{LqLr-bdd}--\eqref{pressure-bdd}.
\end{proof}

We next deduce the asymptotic behavior of $V(t,s)f$ near $t=s$
in some Sobolev spaces when $f\in L^q_\sigma(D_R)\cap W^{1,q}(D_R)$.
It should be emphasized that
$f$ does not satisfy the boundary condition
$f|_{\partial D_R}=0$, and the reason why we have to discuss this case
is related to the function space $Z_q(D)$, see \eqref{auxi},
in which the boundary condition at $\partial D$ is not involved.
In fact, the following lemma plays a role in the proof of 
Lemma \ref{nochmalHR}.
Estimate \eqref{2nd-V} below should be compared with 
\cite[Corollary 4.2]{HR14}, where less singular behavior 
$(t-s)^{-1/2}$ is deduced for
$f\in L^q_\sigma(D_R)\cap W^{1,q}_0(D_R)$ satisfying $f|_{\partial D_R}=0$.
\begin{lemma}
Suppose that $\eta$ and $\omega$ fulfill \eqref{rigid} for 
some $\theta\in (0,1)$.
Let $1<q\leq r<\infty$ and $\delta\in (0,1/2q)$.
For each $\tau_*\in (0,\infty)$ and $m\in (0,\infty)$, there are constants
$C_1=C_1(\tau_*,m,q,\delta,\theta,D_R)>0$ and 
$C_2=C_2(\tau_*,m,q,r,\delta,\theta,D_R)>0$ such that
\begin{equation}
\begin{split}
\|V(t,s)f\|_{W^{2,q}(D_R)}+\|\partial_tV(t,s)f\|_{q,D_R}
&+\|\nabla p(t)\|_{q,D_R}  \\
&\leq C_1(t-s)^{-1+\delta}\|f\|_{W^{1,q}(D_R)}
\end{split}
\label{2nd-V}
\end{equation}
\begin{equation}
\|p(t)\|_{q,D_R}\leq C_1(t-s)^{-(1+1/q)/2+\delta}\|f\|_{W^{1,q}(D_R)}
\label{pre-0}
\end{equation}
\begin{equation}
\|V(t,s)f\|_{W^{1,r}(D_R)}
\leq C_2(t-s)^{-(3/q-3/r)/2-1/2+\delta}\|f\|_{W^{1,q}(D_R)}
\label{1st-V}
\end{equation}
for all $(t,s)\in\Lambda(\tau_*)$ and 
$f\in L^q_\sigma(D_R)\cap W^{1,q}(D_R)$ whenever
\eqref{rigid-bound} is satisfied,
where $\Lambda(\tau_*)$ is given by \eqref{para-set}.
Here, $p(t)$ denotes the pressure associated with $V(t,s)f$ and
it is singled out subject to the side condition
$\int_{D_R}p\,dx=0$.
\label{no-bc}
\end{lemma}

\begin{proof}
As in the proof of \cite[Lemma 3.2]{Hi18},
there is a constant $k=k(m)>0$ such that $k+L_R(t)$ is invertible
in $L^q_\sigma(D_R)$ for all $t\geq 0$ subject to
\[
\sup_{t\geq 0}\|(k+L_R(t))^{-1}\|_{{\cal L}(L^q_\sigma(D_R))}<\infty.
\]
Indeed one can take even $k=0$ by a compactness argument
(see, for instance, \cite[Section 3]{HiS}, \cite[Section 5]{Hi16}),
but this refinement is not needed here.
We then know that
\begin{equation*}
\|L_R(t)V(t,s)f\|_{q,D_R}
\leq C\|(k+L_R(s))f\|_{q,D_R}
\leq C\|f\|_{D_q(A)},
\qquad f\in D_q(A),
\end{equation*}
and
\[
\|L_R(t)V(t,s)f\|_{q,D_R}\leq C(t-s)^{-1}\|f\|_{q,D_R}, \qquad 
f\in L^q_\sigma(D_R),
\]
for all $(t,s)\in \Lambda(\tau_*)$.
In fact, the latter was shown in \cite[(3.20)]{Hi18},
while one verifies the former 
(particularly the first inequality)
if one follows the argument of
general theory \cite[Chapter 5, Theorem 2.1]{T}
under the conditions \eqref{rigid} and \eqref{rigid-bound}.

By complex interpolation we have
\begin{equation*}
\|L_R(t)V(t,s)f\|_{q,D_R}\leq C(t-s)^{-1+\delta}\|f\|_{D_q(A^\delta)}
\end{equation*}
for all $(t,s)\in\Lambda(\tau_*)$ and
\[
f\in D_q(A^\delta)=[L^q_\sigma(D_R),D(A)]_\delta
=L^q_\sigma(D_R)\cap[L^q(D_R),W^{1,q}_0(D_R)\cap W^{2,q}(D_R)]_\delta
\]
where
$[\cdot,\cdot]_\delta$ stands for the complex interpolation functor and
the characterization of $D_q(A^\delta)$ is due to Giga \cite{Gi85}.
As a consequence, we get
\begin{equation}
\begin{split}
&\quad \|V(t,s)f\|_{W^{2,q}(D_R)}+\|\partial_tV(t,s)f\|_{q,D_R}
+\|\nabla p(t)\|_{q,D_R}  \\
&\leq C\|L_R(t)V(t,s)f\|_{q,D_R}+C\|V(t,s)f\|_{q,D_R}  \\
&\leq C(t-s)^{-1+\delta}\|f\|_{D_q(A^\delta)}
\end{split}
\label{2nd-inter}
\end{equation}
for all $(t,s)\in\Lambda(\tau_*)$ and $f\in D_q(A^\delta)$ provided 
$0\leq\delta\leq 1$.

If in particular
$\delta\in (0,1/2q)$, then the space $D_q(A^\delta)$ does not involve
the boundary condition, to be precise,
$D_q(A^\delta)=L^q_\sigma(D_R)\cap H^{2\delta}_q(D_R)$,
where 
$H^{2\delta}_q(D_R):=[L^q(D_R), W^{2,q}(D_R)]_\delta$
is the Bessel potential space,
see Fujiwara \cite[Section 2, Theorem 5]{Fu68}
(this theorem asserts a characterization of some
complex interpolation spaces).
We thus have
$L^q_\sigma(D_R)\cap W^{1,q}(D_R)
\subset D_q(A^\delta)$ for $\delta\in (0,1/2q)$ 
and, therefore, \eqref{2nd-inter}
leads us to \eqref{2nd-V}.

We next observe
\begin{equation}
\|V(t,s)f\|_{W^{1+j,q}(D_R)}\leq C(t-s)^{-j/2}\|f\|_{W^{1,q}(D_R)}
\label{reg-auxi}
\end{equation}
for all $(t,s)\in\Lambda(\tau_*)$,
$f\in L^q_\sigma(D_R)\cap W^{1,q}_0(D_R)$ and $j=0,\, 1$.
In \cite[Corollary 4.2]{HR14} Hansel and Rhandi proved \eqref{reg-auxi}
for such data satisfying $f|_{\partial D_R}=0$
and $0\leq s<t\leq {\cal T}$ with $C>0$ that depends on
${\cal T}\in (0,\infty)$ under the condition \eqref{loc-hoelder},
however, we need to show that the constant $C>0$
can be taken uniformly in $(t,s)\in\Lambda(\tau_*)$ 
as long as \eqref{rigid} is fulfilled.
In fact, using \eqref{2nd-inter} with $\delta=1/2$, we find 
$\mbox{\eqref{reg-auxi}}_{j=1}$
since we know from \cite{Fu68} and \cite{Gi85} that
$D_q(A^{1/2})=L^q_\sigma(D_R)\cap W^{1,q}_0(D_R)$.
We also have \eqref{2nd-inter} with $\delta=1$ as well as 
$\mbox{\eqref{LqLr-bdd}}_{j=0}$
with $r=q$, which implies 
${\eqref{reg-auxi}}_{j=0}$
by interpolation.
The interpolation argument once more by use of
$\mbox{\eqref{reg-auxi}}_{j=0}$ and \eqref{LqLr-bdd} with $r=q$ yields
\[
\|V(t,s)f\|_{W^{1,q}(D_R)}\leq C(t-s)^{-1/2+\delta}\|f\|_{H_q^{2\delta}(D_R)}
\]
for all $(t,s)\in\Lambda(\tau_*)$ and
\[
f\in [L^q_\sigma(D_R),L^q_\sigma(D_R)\cap W^{1,q}_0(D_R)]_{2\delta}
=L^q_\sigma(D_R)\cap H_q^{2\delta}(D_R)
\]
provided $\delta\in (0,1/2q)$,
where the last equality follows from the reiteration theorem 
for the complex interpolation
\cite{BL} combined with the Fujiwara theorem \cite{Fu68} employed above;
thereby, we infer
\begin{equation}
\|V(t,s)f\|_{W^{1,q}(D_R)}\leq C(t-s)^{-1/2+\delta}\|f\|_{W^{1,q}(D_R)}
\label{1st-V0}
\end{equation}
for all $(t,s)\in\Lambda(\tau_*)$ and
$f\in L^q_\sigma(D_R)\cap W^{1,q}(D_R)$.
This together with \eqref{2nd-V} concludes 
\eqref{pre-0} by virtue of \eqref{pre-trace}.

It turns out that
\begin{equation}
\|V(t,s)g\|_{W^{1,r}(D_R)}\leq C(t-s)^{-(3/q-3/r)/2}\|g\|_{W^{1,q}(D_R)}
\label{1st-LqLr}
\end{equation}
for all $(t,s)\in\Lambda(\tau_*)$ and 
$g\in L^q_\sigma(D_R)\cap W^{1,q}_0(D_R)$,
where $1<q\leq r<\infty$.
In fact, this follows from \eqref{reg-auxi}
together with the Gagliardo-Nirenberg
inequality provided that $3/q-3/r\leq 1$.
If $r$ is not close to $q$, then one has only to use the semigroup
property.
Note that $g=T((s+t)/2,s)f$ fulfills the boundary condition
$g|_{\partial D_R}=0$ so that $g\in L^q_\sigma(D_R)\cap W^{1,q}_0(D_R)$
even though $f\in L^q_\sigma(D_R)\cap W^{1,q}(D_R)$.
Hence, by the semigroup property, \eqref{1st-V0} and \eqref{1st-LqLr}
imply \eqref{1st-V}.
The proof is complete.
\end{proof}

\section{Regularity of the evolution operator}
\label{regularity}

Some regularity properties as well as construction of the evolution operator 
$T(t,s)$ were proved by Hansel and Rhandi \cite{HR14}, 
nevertheless, we need more analysis, especially,
\smallskip

\noindent
-- smoothing effect of $T(t,s): Z_q(D)\to Y_q(D)$ when $3/2<q<\infty$;

\noindent
-- smoothing effect of $T(t,s): Z_q(D)\to Z_r(D)$ when $3/2<q< r<\infty$;

\noindent
-- justification of $\partial_tT(t,s)f$ in $W^{-1,q}(D_R)$ for
$f\in L^q_\sigma(D)$ when $1<q<\infty$;
\smallskip

\noindent
which are not covered by \cite{HR14}, where
$Y_q(D)$ and $Z_q(D)$ are defined by \eqref{auxi}.
We will also show the second assertion of Proposition \ref{basic},
that is related to the first issue above since it slightly improves
the corresponding result of \cite{HR14}.
The restriction $q>\frac{3}{2}=\frac{n}{n-1}$ 
($n$ denotes the space dimension) stems from Lemma \ref{proj}
below on some weighted estimate of the Fujita-Kato projection.
The third issue above is quite important to proceed to analysis
of large time behavior of $T(t,s)$.

Let us recall the idea of \cite{HR14} for construction of
a parametrix of the evolution operator by use of
evolution operators in the whole space $\mathbb R^3$ and
in the bounded domain $D_{R_0+6}$, where $R_0$ is as in \eqref{obstacle}.
We fix three cut-off functions
\begin{equation*}
\begin{split}
&\phi\in C_0^\infty(B_{R_0+4}), \qquad \phi=1\quad\mbox{in $B_{R_0+3}$}, \\
&\phi_0\in C_0^\infty(B_{R_0+2}), \qquad \phi=1\quad\mbox{in $B_{R_0+1}$}, \\
&\phi_1\in C_0^\infty(B_{R_0+6}), \qquad \phi=1\quad\mbox{in $B_{R_0+5}$},
\end{split}
\end{equation*}
and set
\begin{equation*}
\begin{split}
&A=\{R_0+2<|x|<R_0+4\}, \quad
A_0=\{R_0<|x|<R_0+2\}, \\
&A_1=\{R_0+4<|x|<R_0+6\}.
\end{split}
\end{equation*}
By $\mathbb B=\mathbb B_A$,
$\mathbb B_0=\mathbb B_{A_0}$ and
$\mathbb B_1=\mathbb B_{A_1}$
we denote the Bogovskii operators, see \eqref{bogov}, in the bounded domains
$A,\, A_0$ and $A_1$, respectively.
Given $f\in L^q_\sigma(D)$, $1<q<\infty$, let us set
\begin{equation*}
\begin{split}
&f_0=(1-\phi_0)f+\mathbb B_0[f\cdot\nabla\phi_0]\in L^q_\sigma(\mathbb R^3), \\
&f_1=\phi_1f-\mathbb B_1[f\cdot\nabla\phi_1]\in L^q_\sigma(D_{R_0+6}),
\end{split}
\end{equation*}
where $f_0$ is understood as its extension to $\mathbb R^3$
by putting zero outside $D$,
then we see from \eqref{bog-est-1} that
\begin{equation}
\begin{split}
&\|f_0\|_{q,\mathbb R^3}+\|f_1\|_{q,D_{R_0+6}}\leq C\|f\|_q, \\
&\|\nabla f_0\|_{q,\mathbb R^3}+\|\nabla f_1\|_{q,D_{R_0+6}}
\leq C\|f\|_{W^{1,q}(D)},  \\
&\||x|\nabla f_0\|_{q,\mathbb R^3}
\leq C\||x|\nabla f\|_q+C\|f\|_q,
\end{split}
\label{data}
\end{equation}
where $\nabla f\in L^q(D)$ is additionally assumed for
$\mbox{\eqref{data}}_2$ and even $|x|\nabla f\in L^q(D)$ is assumed for
$\mbox{\eqref{data}}_3$.
Thus, $f_0\in Z_q(\mathbb R^3)$ follows from $f\in Z_q(D)$.

It is reasonable to start with
\begin{equation}
W(t,s)f
=(1-\phi)U(t,s)f_0+\phi V(t,s)f_1
+\mathbb B[(U(t,s)f_0-V(t,s)f_1)\cdot\nabla\phi]
\label{1st-appro}
\end{equation}
as a fine approximation of the evolution operator,
where $U(t,s)$ is the evolution operator for the whole space
problem (Section \ref{whole}) and $V(t,s)$ is the one
for the interior problem (Section \ref{interior}) over $D_{R_0+6}$.
Note that $W(s,s)f=f$.
In what follows, let us fix $\tau_*\in (0,\infty)$
as well as $m\in (0,\infty)$, and suppose \eqref{rigid-bound}.
By \eqref{LqLr-whole}, \eqref{LqLr-bdd} and
$\mbox{\eqref{data}}_1$ together with \eqref{bog-est-1}, we easily observe
\begin{equation}
\|\nabla^j W(t,s)f\|_q\leq C(t-s)^{-j/2}\|f\|_q
\label{0th-est}
\end{equation}
for all $(t,s)\in\Lambda(\tau_*)$, $j=0,1$ and $f\in L^q_\sigma(D)$
with $C=C(\tau_*,m,q,\theta,D)>0$.
By $p_1$ we denote the pressure associated with $V(t,s)f_1$
for the interior problem over $D_{R_0+6}$,
and it is singled out subject to the side condition
$\int_{D_{R_0+6}}p_1\,dx=0$.
Then the pair of 
\[
u:=W(t,s)f, \qquad p:=\phi p_1
\]
should obey
\begin{equation*}
\begin{split}
\partial_tu
&=\Delta u+(\eta(t)+\omega(t)\times x)\cdot\nabla u-\omega(t)\times u-\nabla p
-K(t,s)f, \\
\mbox{div $u$}&=0,  \\
u|_{\partial D}&=0, \\
u&\to 0\quad\mbox{as $|x|\to\infty$}, \\
u(\cdot,s)&=f,
\end{split}
\end{equation*}
in $D\times (s,\infty)$
(the equation is actually understood in $L^q(D)$ for $f\in Z_q(D)$) with
\begin{equation}
\begin{split}
&\quad K(t,s)f  \\
&=-2\nabla\phi\cdot\nabla(Uf_0-Vf_1)
-\{\Delta\phi+(\eta+\omega\times x)\cdot\nabla\phi\}(Uf_0-Vf_1)  \\
&\quad -(\nabla\phi)p_1
-\mathbb B[(\partial_tUf_0-\partial_tVf_1)\cdot\nabla\phi]
+\Delta\mathbb B[(Uf_0-Vf_1)\cdot\nabla\phi] \\
&\quad +(\eta+\omega\times x)\cdot\nabla \mathbb B[(Uf_0-Vf_1)\cdot\nabla\phi]
-\omega\times\mathbb B[(Uf_0-Vf_1)\cdot\nabla\phi],
\end{split}
\label{remain}
\end{equation}
where we abbreviate $Uf_0=U(t,s)f_0$ and $Vf_1=V(t,s)f_1$.
As in \cite[(5.3)]{HR14},
it follows from 
\eqref{bog-est-1}, \eqref{bog-est-2},
\eqref{LqLr-whole}, \eqref{weak-deri-wh}, \eqref{weak-est-whole},
\eqref{cls-bdd}, \eqref{cls-bdd-p},
Lemma \ref{basic-bdd} and $\mbox{\eqref{data}}_1$
that
\begin{equation}
\begin{split}
&PK(\cdot,s)f\in C((s,\infty); L^q_\sigma(D)), \\
&\|PK(t,s)f\|_q
\leq C(t-s)^{-(1+1/q)/2}\|f\|_q,
\end{split}
\label{est-R0}
\end{equation}
for all $(t,s)\in\Lambda(\tau_*)$ and $f\in L^q_\sigma(D)$ with some
$C=C(\tau_*,m,q,\theta,D)>0$,
where $\Lambda(\tau_*)$ is given by \eqref{para-set}.

The approach adopted by \cite{HR14} is somewhat similar to the one for
construction of parabolic evolution operators,
see \cite[Chapter 5]{T}, although the first approximation
\eqref{1st-appro} is completely different from general theory.
In fact, the idea of \cite{HR14} is to solve the integral equation
\begin{equation}
T(t,s)f=W(t,s)f+\int_s^tT(t,\tau)PK(\tau,s)f\,d\tau.
\label{int-eq}
\end{equation}
To this end, consider the iteration scheme
\begin{equation}
\begin{split}
&T_0(t,s)f=W(t,s)f,  \\
&T_{j+1}(t,s)f=\int_s^tT_j(t,\tau)PK(\tau,s)f\,d\tau \qquad (j=0,1,2,\cdots),
\end{split}
\label{iteration}
\end{equation}
then one can expect that \eqref{sum} below provides a solution as long as it is
convergent.
The argument of \cite{HR14} is based on the following lemma 
on iterated convolutions, see \cite[Lemma 4.6]{GHH}, \cite[Lemma 3.3]{HR11}
and \cite[Lemma 5.2]{HR14}
(the same idea was essentially employed in 
\cite[Chapter 5, Sections 2 and 3]{T}, too).
In those literature the operator families are parametrized by $(t,s)$
with $0\leq s<t\leq {\cal T}$ for fixed ${\cal T}\in (0,\infty)$,
but we need to discuss the ones
parametrized by $(t,s)\in\Lambda(\tau_*)$, see \eqref{para-set},
and what is important is that the constant in \eqref{ite-est} below 
can be taken uniformly in $(t,s)\in\Lambda (\tau_*)$.
This is easily verified by following the proof in the literature above.
\begin{lemma}
[\cite{GHH}, \cite{HR11}, \cite{HR14}]
Let $X_1$ and $X_2$ be two Banach spaces, and fix $\tau_*\in (0,\infty)$.
Suppose that there are constants $\alpha,\,\beta\in [0,1)$ and 
$\kappa>0$ such that
\[
\{A_0(t,s); (t,s)\in\Lambda(\tau_*)\}\subset{\cal L}(X_1,X_2), \qquad
\{Q(t,s); (t,s)\in\Lambda(\tau_*)\}\subset{\cal L}(X_1)
\]
with
\[
\|A_0(t,s)\|_{{\cal L}(X_1,X_2)}\leq \kappa (t-s)^{-\alpha},  \qquad
\|Q(t,s)\|_{{\cal L}(X_1)}\leq \kappa (t-s)^{-\beta}
\]
for all $(t,s)\in\Lambda(\tau_*)$.
For $f\in X_1$ and $(t,s)\in\Lambda(\tau_*)$,
define a sequence $\{A_j(t,s)f\}_{j=0}^\infty \subset X_2$ by
\begin{equation*}
A_{j+1}(t,s)f=\int_s^t A_j(t,\tau)Q(\tau,s)f\,d\tau \qquad (j=0,1,2,\cdots).
\end{equation*}
Then
\[
A(t,s)f:=\sum_{j=0}^\infty A_j(t,s)f \qquad\mbox{in $X_2$}
\]
converges absolutely
and uniformly in $(t,s)\in\Lambda(\tau_*)$ with $t-s\geq\varepsilon$ for every
$\varepsilon \in (0,\tau_*)$.
Moreover, there is a constant $C=C(\tau_*,\kappa,\alpha,\beta)>0$ such that
\begin{equation}
\|A(t,s)f\|_{X_2}
\leq\sum_{j=0}^\infty\|A_j(t,s)f\|_{X_2}
\leq C(t-s)^{-\alpha}\|f\|_{X_1}
\label{ite-est}
\end{equation}
for all $(t,s)\in\Lambda(\tau_*)$ and $f\in X_1$.
If in particular $\alpha=0$, then the convergence of the series above
is uniform in 
$(t,s)\in\overline{\Lambda(\tau_*)}=\{(t,s); 0\leq s\leq t,\,t-s\leq\tau_*\}$.
\label{convo}
\end{lemma}

With \eqref{0th-est} and \eqref{est-R0} at hand,
Hansel and Rhandi \cite{HR14} applied Lemma \ref{convo} with
\[
A_0=W, \quad Q=PK, \quad
X_1=X_2=L^q_\sigma(D), \quad
\alpha=0, \quad 
\beta=\frac{1}{2}\left(1+\frac{1}{q}\right)
\]
to \eqref{iteration}
and succeeded in construction of the evolution operator
\begin{equation}
T(t,s)f:=\sum_{j=0}^\infty T_j(t,s)f
\label{sum}
\end{equation}
which solves \eqref{int-eq}.
This was quite successful.
In order to show that $T(t,s)$ leaves $Y_q(D)$ invariant,
they first intended to prove
$T(t,s)Z_{q,0}(D)\subset Z_{q,0}(D)$, 
where $Z_{q,0}(D)=\{f\in Z_q(D);\, f|_{\partial D}=0\}$, see \eqref{auxi}.
Note that $Z_{q,0}(D)$ is denoted by $Z$ in their paper, see
\cite[p.17]{HR14}.
To this end, they applied Lemma \ref{convo} with
$X_1=X_2=Z_{q,0}(D)$ as well as $A_0=W$ and $Q=PK$,
however, $PK(t,s)f$ cannot always belong to $Z_{q,0}(D)$ 
because $PK(t,s)f$ does not satisfy the homogeneous Dirichlet
boundary condition at $\partial D$ no matter how fine $f$ is.
Indeed this is unfortunately an oversight of \cite{HR14},
but their argument can be corrected in the following way.

The idea of correction is to replace $Z_{q,0}(D)$ by $Z_q(D)$,
which does not involve the homogeneous Dirichlet boundary condition,
and to employ the following weighted estimate of the Fujita-Kato projection.
For the weighted estimate,
one needs the restriction $q\in (3/2,\infty)$, however, 
this is not an obstacle for later argumant.
See \cite[Proposition 4.3]{Hi99} for similar consideration
in the case $q=2$.
Note that the following lemma holds true even for $g\in W^{1,q}(D)$
(without boundary condition) with
$|x|\nabla g\in L^q(D)$
if the second term of the RHS of \eqref{weighted} is replaced by
$\|\nabla g\|_q$.
Since we will use this lemma only with $g=K(t,s)f$, see \eqref{remain},
it is given in the following form.
\begin{lemma}
Let $3/2<q<\infty$.
Then there is a constant $C=C(q,D)>0$ such that
\begin{equation}
\||x|\nabla Pg\|_q\leq C(\||x|\nabla g\|_q+\|\mbox{\rm{div} $g$}\|_q +\|g\|_q)
\label{weighted}
\end{equation}
for all $g\in W_0^{1,q}(D)^3$ with $|x|\nabla g\in L^q(D)^{3\times 3}$.
\label{proj}
\end{lemma}

\begin{proof}
Consider the Neumann problem
\[
-\Delta w=\mbox{div $g$}\quad\mbox{in $D$}, \qquad
\frac{\partial w}{\partial\nu}\Big|_{\partial D}
=-\nu\cdot g|_{\partial D}=0,
\]
where $\nu$ stands for the outer unit normal to $\partial D$.
It then suffices to show
\begin{equation}
\||x|\nabla^2 w\|_q
\leq C(\||x|(\mbox{div $g$})\|_q+\|\mbox{div $g$}\|_q +\|g\|_q)
\label{neumann}
\end{equation}
which implies \eqref{weighted} since
$Pg=g+\nabla w$.
We fix $L\in (R_0,\infty)$ and take a cut-off function
$\phi\in C_0^\infty(D_L)$ such that $\phi=1$ in $B_{R_0}$,
where $R_0$ is as in \eqref{obstacle}.
We choose a solution $w$ satisfying
$\int_{D_L}w\,dx=0$, so that
\begin{equation}
\|w\|_{q,D_L}\leq C\|\nabla w\|_{q,D_L}
\leq C\|\nabla w\|_q
\leq C\|g\|_q
\label{neumann-est}
\end{equation}
where the last inequality is due to \cite{Mi}, \cite{SiS}.
Then $\phi w$ obeys
\[
-\Delta(\phi w)
=\phi(\mbox{div $g$})-2\nabla\phi\cdot\nabla w-(\Delta\phi)w \quad
\mbox{in $D_L$}, \qquad
\nu\cdot\nabla(\phi w)|_{\partial D_L}=0,
\]
which leads to
\begin{equation}
\|\nabla^2(\phi w)\|_{q,D_L}
\leq C\|\mbox{div $g$}\|_q+C\|w\|_{W^{1,q}(D_L)},
\label{cut-bdd}
\end{equation}
where, this time, $\nu$ denotes the outer unit normal to $\partial D_L$.
On the other hand, $(1-\phi)w$ obeys
\[
-\Delta\{(1-\phi)w\}
=(1-\phi)(\mbox{div $g$})+2\nabla\phi\cdot\nabla w+(\Delta\phi)w=:h \quad
\mbox{in $\mathbb R^3$}.
\]
By ${\cal R}=\nabla(-\Delta)^{-1/2}$ we denote the Riesz transform,
then we know
\[
\||x|{\cal R}h\|_{q,\mathbb R^3}
\leq C\||x|h\|_{q,\mathbb R^3}
\]
from the Muckenhoupt theory for singular integrals as long as
$\frac{n}{n-1}=\frac{3}{2}<q<\infty$;
in fact, for such $q$, the weight
$|x|^q$ belongs to the Muckenhoupt class ${\cal A}_q(\mathbb R^3)$,
see Farwig and Sohr \cite[Section 2]{FaS},
Stein \cite[Chapter V]{St},
Torchinsky \cite[Chapter IX]{To} for details.
We thus obtain
\begin{equation}
\begin{split}
\||x|\nabla^2\{(1-\phi)w\}\|_{q,\mathbb R^3}
&=\||x|({\cal R}\otimes{\cal R})h\|_{q,\mathbb R^3}
\leq C\||x|h\|_{q,\mathbb R^3}  \\
&\leq C\||x|(\mbox{div $g$})\|_q+C\|w\|_{W^{1,q}(D_L)}
\end{split}
\label{cut-wh}
\end{equation}
for $3/2<q<\infty$.
We collect \eqref{neumann-est}, \eqref{cut-bdd} and \eqref{cut-wh}
to conclude \eqref{neumann}.
\end{proof}

Since the functions being in our class $Z_q(D)$ do not satisfy
the Dirichlet boundary condition, we have to replace
\cite[(5.4)]{HR14} by \eqref{est-R} of the following lemma.
The smoothing rate $(t-s)^{-1+\delta}$ below stems from \eqref{2nd-V} 
for the interior problem.
\begin{lemma}
Suppose that $\eta$ and $\omega$ fulfill \eqref{rigid} for some 
$\theta\in (0,1)$.
Let $1<q<\infty$ and $\delta\in (0,1/2q)$.
Given $\tau_*\in (0,\infty)$ and $m\in (0,\infty)$, let
$\Lambda(\tau_*)$ be as in \eqref{para-set}
and assume \eqref{rigid-bound}.

\begin{enumerate}
\item
There is a constant $C=C(\tau_*,m,q,\delta,\theta,D)>0$
such that, for every $f\in Z_q(D)$ and $t\in (s,\infty)$,
we have $W(t,s)f\in Y_q(D)$ subject to
\begin{equation}
\|W(t,s)f\|_{Y_q(D)}\leq C(t-s)^{-1+\delta}\|f\|_{Z_q(D)}
\label{est-YZ}
\end{equation}
for all $(t,s)\in\Lambda(\tau_*)$.

\item
There is a constant $C=C(\tau_*,m,q,\delta,\theta,D)>0$
such that
\begin{equation}
\|K(t,s)f\|_{W^{1,q}(D)}\leq C(t-s)^{-1+\delta}\|f\|_{Z_q(D)}
\label{est-R}
\end{equation}
for all $(t,s)\in\Lambda(\tau_*)$ and $f\in Z_q(D)$.
If in particular $q\in (3/2,\infty)$, then there is a constant
$C=C(\tau_*,m,q,\delta,\theta,D)>0$ such that
\begin{equation}
\|PK(t,s)f\|_{Z_q(D)}\leq C(t-s)^{-1+\delta}\|f\|_{Z_q(D)}
\label{est-PR}
\end{equation}
for all $(t,s)\in\Lambda(\tau_*)$ and $f\in Z_q(D)$.
\end{enumerate}
\label{nochmalHR}
\end{lemma}

\begin{proof}
We collect \eqref{bog-est-1}, \eqref{LqLr-whole},
\eqref{est-YZ-whole}, \eqref{LqLr-bdd}, \eqref{2nd-V},
\eqref{pre-0} and \eqref{data}
to obtain \eqref{est-YZ} and \eqref{est-R}.
Since $K(t,s)f\in W^{1,q}_0(D)$ with
$|x|\nabla K(t,s)f\in L^q(D)$,
one can use \eqref{weighted} to obtain \eqref{est-PR}.
\end{proof}

\noindent
{\it Proof of the second assertion of Proposition \ref{basic}}.
Let $3/2<q<\infty$.
In view of \eqref{iteration}, \eqref{est-YZ} and \eqref{est-PR} one can apply 
Lemma \ref{convo} with
\[
A_0=W, \quad Q=PK, \quad
X_1=Z_q(D), \quad X_2=Y_q(D), \quad 
\alpha=\beta=1-\delta
\]
to see that
$T(t,s)f\in Y_q(D)$ with
\begin{equation}
\|T(t,s)f\|_{Y_q(D)}\leq C(t-s)^{-1+\delta}\|f\|_{Z_q(D)}
\label{T-YZ}
\end{equation}
for all $(t,s)\in\Lambda(\tau_*)$ and $f\in Z_q(D)$.
Note that \cite[(5.9)]{HR14} is now replaced by \eqref{T-YZ}.
The proof of the other parts by \cite{HR14} is correct
and there is no need to repeat it.
Here, the assertion has been proved under the condition \eqref{rigid}
in order to deduce all the estimates with constants uniformly in
$(t,s)\in\Lambda(\tau_*)$;
in fact, such estimates are needed for Proposition \ref{ann-18}.
But one can show Proposition \ref{basic} under the same condition
\eqref{loc-hoelder} as in \cite{HR14} subject to the corresponding estimates
for $0\leq s<t\leq {\cal T}$, where ${\cal T}\in (0,\infty)$
is arbitrarily fixed.
\hfill
$\Box$
\medskip

The following lemma on smoothing effect in the framework of the space 
$Z_q(D)$ is needed in the proof
of Proposition \ref{1st-LED}.
\begin{lemma}
Suppose that $\eta$ and $\omega$ fulfill \eqref{rigid} for some $\theta\in (0,1)$.
Let $3/2<q\leq r<\infty$.
For every $f\in Z_q(D)$ and $t\in (s,\infty)$, we have
$T(t,s)f\in Z_r(D)$.
\label{smoothing-Z}
\end{lemma}

\begin{proof}
Let $\tau_*\in (0,\infty)$ and $\delta\in (0,1/2q)$.
By \eqref{bog-est-1}, \eqref{est-Z-whole} and \eqref{1st-V} 
together with \eqref{data} we find
\[
\|W(t,s)f\|_{Z_r(D)}\leq C(t-s)^{-(3/q-3/r)/2-1/2+\delta}\|f\|_{Z_q(D)}
\]
for all $(t,s)\in\Lambda(\tau_*)$ and $f\in Z_q(D)$ even if
$1<q\leq r<\infty$.
By virtue of this combined with \eqref{est-PR}, we apply Lemma \ref{convo} with
\begin{equation*}
\begin{split}
&A_0=W, \quad Q=PK, \quad
X_1=Z_q(D), \quad X_2=Z_r(D), \\
&\alpha=\frac{3}{2}\left(\frac{1}{q}-\frac{1}{r}\right)+\frac{1}{2}-\delta,
\quad \beta=1-\delta
\end{split}
\end{equation*}
to get the conclusion subject to
\[
\|T(t,s)f\|_{Z_r(D)}\leq C(t-s)^{-\alpha}\|f\|_{Z_q(D)}
\]
for all $(t,s)\in\Lambda(\tau_*)$ and $f\in Z_q(D)$ provided
$\alpha<1$ as well as $q\in (3/2,\infty)$.
The condition $\alpha <1$ with some $\delta\in (0,1/2q)$ is always
accomplished
for every $r\in [q,\infty)$ when $q\geq 2$.
Otherwise ($3/2<q<2$), one needs a restriction that $r$ is not too large.
In this latter case, $T(t,s)f\in Z_r(D)$ for $r\in (q,2]$ is
always possible and then we have only to use the semigroup property to obtain
$T(t,s)f\in Z_r(D)$ even for $r\in (2,\infty)$ as follows:
\begin{equation*}
\begin{split}
\|T(t,s)f\|_{Z_r(D)}
&\leq C(t-s)^{-(3/2-3/r)/2-1/2+\widetilde\delta}
\|T((t+s)/2,s)f\|_{Z_2(D)}  \\
&\leq C(t-s)^{-(3/q-3/r)/2-1+\widetilde\delta+\delta}\|f\|_{Z_q(D)}
\end{split}
\end{equation*}
for all $(t,s)\in\Lambda(\tau_*)$ and $f\in Z_q(D)$,
where $\max\{1/4-3/2r,\,0\}<\widetilde\delta<1/4$ and $\delta\in (0,1/2q)$.
The proof is complete.
\end{proof}

The following result justifies the derivative 
with respect to time variable with values in
$W^{-1,q}(D_R)$ for general data being in $L^q_\sigma(D)$.
This is indeed a key observation in the present paper and
can be regarded as a substitution of \cite[Theorem 5.1]{HiS}
for autonomous case.
Here, a bounded domain $D_R$ can be independent of $D_{R_0+6}$ in which the 
solution
$V(t,s)f_1$ was found in constructing the parametrix \eqref{1st-appro}.
\begin{proposition}
Suppose that $\eta$ and $\omega$ fulfill \eqref{rigid} for some 
$\theta\in (0,1)$.
Let $1<q<\infty$ and $R\in (R_0+1,\infty)$,
where $R_0$ is as in \eqref{obstacle}.
Given $f\in L^q_\sigma(D)$, we set $u(t)=T(t,s)f$.
Given $\tau_*\in (0,\infty)$ and $m\in (0,\infty)$, let
$\Lambda(\tau_*)$ be as in \eqref{para-set}
and assume
\eqref{rigid-bound}.

\begin{enumerate}
\item
There is a constant $C=C(\tau_*,m,q,R,\theta,D)>0$ such that
\begin{equation}
u\in C^1((s,\infty); W^{-1,q}(D_R)),
\label{weak-deri}
\end{equation}
\begin{equation}
\|\partial_tT(t,s)f\|_{W^{-1,q}(D_R)}\leq C(t-s)^{-(1+1/q)/2}\|f\|_q
\label{weak-est}
\end{equation}
for all $(t,s)\in\Lambda(\tau_*)$ and $f\in L^q_\sigma(D)$.
Furthermore, we have the pressure $p(t)$ subject to
$\int_{D_R}p\,dx=0$ such that the pair of $\{u,p\}$ satisfies
\begin{equation}
\begin{split}
\langle\partial_tu, \psi\rangle_{D_R}
+\langle \nabla u+u\otimes (\eta+\omega\times x)
&-(\omega\times x)\otimes u, \nabla\psi\rangle_{D_R}   \\
&-\langle p, \mbox{\rm{div} $\psi$}\rangle_{D_R}
=0
\end{split}
\label{weak-eqn}
\end{equation}
for all $t\in (s,\infty)$ and $\psi\in W^{1,q^\prime}_0(D_R)^3$, that
\begin{equation}
\|p(t)\|_{q,D_R}
\leq C\|\partial_tu(t)\|_{W^{-1,q}(D_R)}
+C\|u(t)\|_{W^{1,q}(D_R)}
\label{pressure-cont}
\end{equation}
for all $t\in (s,\infty)$
with a constant $C=C(m,q,R,D)>0$
and that
\begin{equation}
\|p(t)\|_{q,D_R}\leq C(t-s)^{-(1+1/q)/2}\|f\|_q
\label{pressure-reg}
\end{equation}
for all $(t,s)\in\Lambda(\tau_*)$
with a constant
$C=C(\tau_*,m,q,R,\theta,D)>0$,
where both constants above are independent of
$f\in L^q_\sigma(D)$.

\item
If in particular $q\in (3/2,\infty)$ and $f\in Z_q(D)$, then
there is a constant $C=C(\tau_*,m,q,R,\theta,D)>0$ such that
\begin{equation}
\|L_+(t)T(t,s)f\|_{W^{-1,q}(D_R)}
\leq C(t-s)^{-(1+1/q)/2}\|f\|_q
\label{weak-est2}
\end{equation}
for all $(t,s)\in\Lambda(\tau_*)$.
\end{enumerate}
\label{weak}
\end{proposition}

\begin{proof}
From Lemma \ref{weak-whole}, \eqref{cls-bdd} and Lemma \ref{basic-bdd} 
we infer that
\[
W(\cdot,s)f\in C^1((s,\infty); W^{-1,q}(D_R))
\]
with
\[
\|\partial_tW(t,s)f\|_{W^{-1,q}(D_R)}\leq C(t-s)^{-(1+1/q)/2}\|f\|_q
\]
for all $(t,s)\in\Lambda(\tau_*)$ and $f\in L^q_\sigma(D)$.
Here, notice that
\[
\partial_t\mathbb B[(U(t,s)f)\cdot\nabla\phi]
=\mathbb B[(\partial_tU(t,s)f)\cdot\nabla\phi]
\]
holds even in $L^q(D_R)$, which follows from \eqref{bog-est-2}
and \eqref{weak-deri-wh}.
Starting from $W(t,s)f$ together with \eqref{est-R0},
we use \eqref{iteration} to show by induction that
\[
T_j(\cdot,s)\in C^1((s,\infty); W^{-1,q}(D_R))
\]
for every $j$ with
\begin{equation*}
\begin{split}
&\partial_tT_0(t,s)f=\partial_tW(t,s)f, \\
&\partial_tT_1(t,s)f=PK(t,s)f+\int_s^t\partial_tW(t,\tau)PK(\tau,s)f\,d\tau, \\
&\partial_tT_{j+1}(t,s)f=\int_s^t\partial_tT_j(t,\tau)PK(\tau,s)f\,d\tau
\qquad (j=1,2,\cdots).
\end{split}
\end{equation*}
and that
\begin{equation}
\|\partial_tT_j(t,s)f\|_{W^{-1,q}(D_R)}\leq \mu_j(t-s)^{-(1+1/q)/2}\|f\|_q
\label{est-Tj-deri}
\end{equation}
for all $(t,s)\in\Lambda(\tau_*)$ and $f\in L^q_\sigma(D)$
with
\[
\mu_j=\mu_j(\tau_*,m,q,R,\theta,D)
=\frac{c_0c_1^j}{\Gamma((1-\alpha)j)} \qquad (j=1,2,\cdots)
\]
where $\alpha=(1+1/q)/2$, $\Gamma(\cdot)$ denotes the Gamma function,
and positive constants $c_0,\, c_1$ are independent of $j$,
so that
$\sum_{j=1}^\infty\mu_j<\infty$.
This can be verified along the same way as in the proof of Lemma \ref{convo},
see \cite[Lemma 3.3]{HR11}, \cite[Chapter 5, Section 2]{T}.
Hence, for each $s\geq 0$, the series
$\sum_{j=0}^\infty \partial_tT_j(t,s)f$
converges in $W^{-1,q}(D_R)$
uniformly with respect to $t\in [s+\varepsilon, s+\tau_*]$
for every $\varepsilon\in (0,\tau_*)$.
We thus conclude \eqref{weak-deri} with
\[
\partial_tT(t,s)f=\sum_{j=0}^\infty \partial_tT_j(t,s)f
\]
in $W^{-1,q}(D_R)$, which yields \eqref{weak-est}.
This combined with the second assertion of Proposition \ref{basic}
implies \eqref{weak-est2} as well.
Formally, the result obtained here is observed by applying 
Lemma \ref{convo} with
\begin{equation*}
\begin{split}
&A_0=\partial_tT_1, \quad Q=PK, \quad
X_1=L^q_\sigma(D),\quad  X_2=W^{-1,q}(D_R), \\
&\alpha=\beta=\frac{1}{2}\left(1+\frac{1}{q}\right),
\end{split}
\end{equation*}
however, the differentiability of $T_j(t,s)f$ with respect to $t$
is verified simultaneously with \eqref{est-Tj-deri}; thus, we should 
take the way explained above.

Suppose $f\in C_{0,\sigma}^\infty(D)$ and set $u(t)=T(t,s)f$.
By $p(t)$ we denote the associated pressure which is singled out such that
$\int_{D_R}p\,dx=0$.
Combining the equation \eqref{IVP} with
\[
\|p(t)\|_{q,D_R}\leq C\|\nabla p(t)\|_{W^{-1,q}(D_R)}
\]
(see, for instance, \cite[Remark 4.1]{Hi04} for its proof
with the aid of \eqref{bog-est-1}),
we find \eqref{pressure-cont} for $f\in C_{0,\sigma}^\infty(D)$
as well as $p\in C((s,\infty); L^q(D_R))$.
Thus, \eqref{pressure-reg} follows from
\eqref{weak-est} together with
the second assertion of Proposition \ref{ann-18}
when $f\in C_{0,\sigma}^\infty(D)$.
We next take general $f\in L^q_\sigma(D)$, then by approximation
we get the function $p_R\in C((s,\infty); L^q(D_R))$ 
which together with $u(t)=T(t,s)f$
enjoys \eqref{weak-eqn} as well as the
same estimates \eqref{pressure-cont}--\eqref{pressure-reg} 
and $\int_{D_R}p_R\,dx=0$.
In this way, for every integer $k>0$, we obtain the pressure
$p_{R+k}$ over $D_{R+k}$ satisfying $\int_{D_{R+k}}p_{R+k}\,dx=0$, however,
we see from \eqref{weak-eqn} that
\[
\langle p_{R+k}(t)-p_{R+j}(t),\, \mbox{div $\psi$}\rangle_{D_{R+j}}=0
\]
for every $\psi\in C_0^\infty(D_{R+j})^3$ and $k>j\geq 0$.
Consequently, $p_{R+k}(x,t)-p_R(x,t)=c_k(t)$ a.e.$D_R$
with some $c_k(t)$ independent of $x\in D_R$.
Let us define
\[
p(x,t)=
\left\{
\begin{array}{ll}
p_R(x,t), & x\in D_R, \\
p_{R+k}(x,t)-c_k(t), \qquad & x\in D_{R+k}\setminus D_{R+k-1}
\quad (k=1,2, \cdots),
\end{array}
\right.
\]
which is the desired pressure over $D$ satisfying
\[
p\in C((s,\infty); L^q(D_R)), \qquad
\int_{D_R}p\,dx=0
\]
as well as \eqref{weak-eqn}--\eqref{pressure-reg}
for all $f\in L^q_\sigma(D)$.
\end{proof}

Analysis in this section can be also carried out
for the evolution operator
$\widetilde T(\tau,s;t)$ generated by
the initial value problem \eqref{ivp-adj} with use of 
$\widetilde U(\tau,s;t)$ given by \eqref{evo-adj-wh} and
the corresponding evolution operator in the bounded domain $D_{R_0+6}$.
Although the latter one is not explicitly given, we do have it
by the Tanabe-Sobolevskii theory \cite[Chapter 5]{T} and it possesses the same
properties as described in Section \ref{interior}.
All the constants in several key estimates can be independent of $t$
and taken uniformly in $(\tau,s)$ with
$\tau -s\leq\tau_*$ as well as $0\leq s<\tau\leq t$.
In view of the relations \eqref{back-op} and \eqref{duality},
the corresponding results for the adjoint $T(t,s)^*$,
especially \eqref{weak-adj0} and \eqref{weak-adj} in the next section, 
are available.

\section{Local energy decay of the evolution operator}
\label{local}

In this section we deduce local energy decay estimates
of the evolution operator: Proposition \ref{1st-LED}
for initial velocity with bounded support and
Proposition \ref{2nd-LED} for general data.
The former is a step to get the latter.
In Proposition \ref{1st-LED} we have a bit less sharp rate of decay
than the desired one $(t-s)^{-3/2}$, but this does not cause any problem.
If we took the same way for general data
as in the proof of Proposition \ref{1st-LED}, we would obtain
less decay rate $(t-s)^{-3/2q+\varepsilon}$ than the one in
Proposition \ref{2nd-LED}.
This never implies Theorem \ref{main}, and thus we should take the
following way.
\begin{proposition}
Suppose that $\eta$ and $\omega$ fulfill \eqref{rigid} for some 
$\theta\in (0,1)$.
Let $R\in (R_0+1,\infty)$, where $R_0$ is as in \eqref{obstacle}. 
Let $\varepsilon >0$ be arbitrarily small.

\begin{enumerate}
\item
Let $1<q<\infty$.
For each $m\in (0,\infty)$, there is a constant
$C=C(m,\varepsilon,q,R,\theta,D)>0$ such that
\begin{equation}
\begin{split}
\|T(t,s)f\|_{W^{1,q}(D_R)}
&\leq C(t-s)^{-3/2+\varepsilon}\|f\|_q, \\
\|T(t,s)^*g\|_{W^{1,q}(D_R)}
&\leq C(t-s)^{-3/2+\varepsilon}\|g\|_q,
\end{split}
\label{1st-led1}
\end{equation}
for all $(t,s)$ with 
\[
t-s>2 \quad\mbox{as well as $0\leq s<t$}
\]
and
$f,\,g\in L^q_\sigma(D)$ with
\[
f(x)=0, \quad g(x)=0 \quad{\rm a.e.}\;\mathbb R^3\setminus B_{3R_0}
\]
whenever
\eqref{rigid-bound} is satisfied.

\item
Let $3/2<q<\infty$.
For each $m\in (0,\infty)$, there is a constant
$C=C(m,\varepsilon,q,R,\theta,D)>0$ such that
\begin{equation}
\begin{split}
\|\partial_tT(t,s)f\|_{W^{-1,q}(D_R)}
&\leq C(t-s)^{-3/2+\varepsilon}\|f\|_q, \\
\|\partial_sT(t,s)^*g\|_{W^{-1,q}(D_R)}
&\leq C(t-s)^{-3/2+\varepsilon}\|g\|_q,
\end{split}
\label{1st-led2}
\end{equation}
for all $(t,s)$ with
\[
t-s>2 \quad\mbox{as well as $0\leq s<t$}
\]
and $f,\,g\in L^q_\sigma(D)\cap W^{1,q}(D)$ with 
\[
f(x)=0, \quad g(x)=0 \quad{\rm a.e.}\;\mathbb R^3\setminus B_{3R_0}
\]
whenever \eqref{rigid-bound} is satisfied.
\end{enumerate}
\label{1st-LED}
\end{proposition}

\begin{proof}
Given $\varepsilon>0$ arbitrarily small as well as $q\in (1,\infty)$,
let us take $p_0$ and $q_0$ such that
$1<p_0<q<q_0<\infty$ and
$(3/p_0-3/q_0)/2=3/2-\varepsilon$.

Set $u(t)=T(t,s)f$ and suppose $t-s>2$.
By both assertions in Proposition \ref{ann-18} we find
\[
\|\nabla T(t,t-1)u(t-1)\|_{q_0,D_R}
\leq C\|u(t-1)\|_{q_0}
\leq C(t-s-1)^{-3/2+\varepsilon}\|f\|_{p_0}
\]
which implies \eqref{1st-led1} for $\nabla T(t,s)f$.
As for $T(t,s)f$ itself (without derivative),
the argument is straightforward without using
semigroup property.

To show the second assertion for $\partial_tu(t)$, we note that
$f\in Z_q(D)$ and, thereby,
$\partial_tu(t)=-L_+(t)u(t)$
provided $q>3/2$, see Proposition \ref{basic}.
By Lemma \ref{smoothing-Z} we know that
$T(t-1,s)f\in Z_{q_0}(D)$ for every $q_0\in (q,\infty)$ and
$t\in (s+1,\infty)$.
It then follows from \eqref{LqLr-noch} with \eqref{weak-est2} that
\begin{equation*}
\begin{split}
\|L_+(t)T(t,t-1)u(t-1)\|_{W^{-1,q_0}(D_R)}
&\leq C\|u(t-1)\|_{q_0}  \\
&\leq C(t-s-1)^{-3/2+\varepsilon}\|f\|_{p_0}
\end{split}
\end{equation*}
which proves \eqref{1st-led2} for $\partial_tT(t,s)f$.

Set $v(s)=T(t,s)^*g$, then we have $v(s)=T(s+1,s)^*v(s+1)$ by the
backward semigroup property.
We then take the same way as above; to be sure, we just describe
several lines only for \eqref{1st-led2}.
As mentioned at the end of the previous section, we have
\begin{equation}
\|L_-(s)T(t,s)^*g\|_{W^{-1,q}(D_R)}
\leq C(t-s)^{-(1+1/q)/2}\|g\|_q
\label{weak-adj0}
\end{equation}
for all $(t,s)\in\Lambda(\tau_*)$ and $g\in Z_q(D)$ with $q\in (3/2,\infty)$,
which corresponds to \eqref{weak-est2} for $T(t,s)$.
Furthermore, similarly to Lemma \ref{smoothing-Z},
we have
\[
v(s+1)=T(t,s+1)^*g=\widetilde T(t-s-1,0;t)g\in Z_{q_0}(D)
\]
for every $q_0\in (q,\infty)$, see \eqref{back-op}.
Therefore, we combine \eqref{weak-adj0} with \eqref{LqLr-noch} to obtain
\begin{equation*}
\begin{split}
\|L_-(s)T(s+1,s)^*v(s+1)\|_{W^{-1,q_0}(D_R)}
&\leq C\|v(s+1)\|_{q_0}  \\
&\leq C(t-s-1)^{-3/2+\varepsilon}\|g\|_{p_0}
\end{split}
\end{equation*}
which leads to \eqref{1st-led2} for $\partial_sT(t,s)^*g$.
\end{proof}

Let us proceed to the second stage of the local energy decay properties,
in which we intend to estimate the evolution operator still
over the bounded domain $D_R$
near the boundary for general data being in $f\in L^q_\sigma(D)$.
\begin{proposition}
Suppose that $\eta$ and $\omega$ fulfill \eqref{rigid} for some 
$\theta\in (0,1)$.
Let $R\in (R_0+1,\infty)$, where $R_0$ is as in \eqref{obstacle}.
Let $1<q<\infty$.
For each $m\in (0,\infty)$, there is a constant 
$C=C(m,q,R,\theta,D)>0$ such that
\begin{equation}
\begin{split}
\|T(t,s)f\|_{W^{1,q}(D_R)}+\|\partial_tT(t,s)f\|_{W^{-1,q}(D_R)}
&\leq C(t-s)^{-3/2q}\|f\|_q, \\
\|T(t,s)^*g\|_{W^{1,q}(D_R)}+\|\partial_sT(t,s)^*g\|_{W^{-1,q}(D_R)}
&\leq C(t-s)^{-3/2q}\|g\|_q,
\end{split}
\label{2nd-led}
\end{equation}
for all $(t,s)$ with
\[
t-s>2 \quad\mbox{as well as $0\leq s<t$}
\]
and $f,\, g\in L^q_\sigma(D)$ whenever \eqref{rigid-bound} is satisfied.
Here, the temporal derivatives are understood as in Proposition \ref{weak}.
\label{2nd-LED}
\end{proposition}

\begin{proof}
By \eqref{LqLr-noch}, \eqref{LqLr-grad-noch} and
\eqref{weak-est} it suffices to prove
\eqref{2nd-led} for all $f,\,g\in C^\infty_{0,\sigma}(D)$.
Concerning the temporal derivatives 
$\partial_tT(t,s)f$ and $\partial_sT(t,s)^*g$,
it is also sufficient to show the assertion for $q\in (3/2,\infty)$;
in fact, once we have that for such $q$ (for instance, $q=3$),
\eqref{LqLr-noch} yields
\begin{equation*}
\begin{split}
\|\partial_tT(t,s)f\|_{W^{-1,q}(D_R)}
&\leq C\|L_+(t)T(t,s)f\|_{W^{-1,3}(D_R)} \\
&\leq C(t-s)^{-1/2}\|T((t+s)/2,s)f\|_3  \\
&\leq C(t-s)^{-3/2q}\|f\|_q
\end{split}
\end{equation*}
even if $q\in (1,3/2]$.

As in the previous study \cite[Section 4]{Hi18}, given 
$f\in C^\infty_{0,\sigma}(D)\subset C^\infty_{0,\sigma}(\mathbb R^3)$,
we regard the solution $T(t,s)f$ as the perturbation from a modification of the $\mathbb R^3$-flow
$U(t,s)f$ as follows:
\[
T(t,s)f=(1-\phi)U(t,s)f+\mathbb B[(U(t,s)f)\cdot\nabla\phi]+v(t),
\]
where $v(t)$ denotes the perturbation,
$\phi\in C_0^\infty(B_{3R_0})$ is a cut-off function satisfying
$\phi=1$ on $B_{2R_0}$ and
$\mathbb B=\mathbb B_{A_{R_0}}$ is the Bogovskii
operator on the domain $A_{R_0}=B_{3R_0}\setminus\overline{B_{R_0}}$,
see \eqref{bogov}.
From $L^q$-$L^\infty$ estimate \eqref{LqLr-whole} and \eqref{bog-est-1}
(together with the equation \eqref{eqn-wh} for $\partial_tU(t,s)f$),
it follows that
$(1-\phi)U(t,s)f+\mathbb B[(U(t,s)f)\cdot\nabla\phi]$ and its 
temporal derivative (even in $L^q(D_R)$) possess the desired decay rate
$(t-s)^{-3/2q}$.
Our task is thus to estimate
\begin{equation}
v(t)=T(t,s)\widetilde f+\int_s^tT(t,\tau)F(\tau)\,d\tau
\label{duha-1}
\end{equation}
and
\begin{equation}
\partial_tv(t)
=\partial_tT(t,s)\widetilde f+F(t)+\int_s^t\partial_tT(t,\tau)F(\tau)\,d\tau
\label{duha-2}
\end{equation}
where
$\widetilde f=\phi f-\mathbb B[f\cdot\nabla\phi]$
and
\begin{equation*}
\begin{split}
F(x,t)
&=-2\nabla\phi\cdot\nabla U(t,s)f
-[\Delta\phi+(\eta(t)+\omega(t)\times x)\cdot\nabla\phi]U(t,s)f  \\ 
&\quad -\mathbb B[(\partial_tU(t,s)f)\cdot\nabla\phi]
+\Delta\mathbb B[(U(t,s)f)\cdot\nabla\phi] \\ 
&\quad +(\eta(t)+\omega(t)\times x)\cdot
\nabla \mathbb B[(U(t,s)f)\cdot\nabla\phi] \\ 
&\quad -\omega(t)\times\mathbb B[(U(t,s)f)\cdot\nabla\phi].
\end{split}
\end{equation*}
The forcing term $F$ fulfills, see \cite[(4.2)]{Hi18},
\begin{equation}
\|F(t)\|_q
\leq
C(m+1)\|f\|_q
\left\{
\begin{array}{ll}
(t-s)^{-1/2}, & 0<t-s<1, \\
(t-s)^{-3/2q}, \qquad &t-s\geq 1,
\end{array}
\right.
\label{remainder-est}
\end{equation}
as well as $\mbox{div $F$}=\Delta p=0$ (so that $PF=F$) which follows at once
from the equation that $\{v,p\}$ obeys, 
where $p$ is the pressure associated with
$T(t,s)f$.
Given $q\in (1,\infty)$, let us take $\varepsilon >0$ so small
that $3/2-\varepsilon >3/2q$.
Suppose $t-s>2$.
By Proposition \ref{1st-LED} with such $\varepsilon$ and
by \eqref{remainder-est}
it is seen that $T(t,s)\widetilde f$ and $\partial_tT(t,s)\widetilde f+F(t)$ 
satisfy the desired decay property.

Let us consider the last terms of \eqref{duha-1}--\eqref{duha-2}.
Concerning the latter one for the temporal derivative
we can apply \eqref{1st-led2}
since $F\in L^q_\sigma(D)\cap W^{1,q}(D)$
with $F(x,t)=0$ a.e. $|x|\geq 3R_0$
(note that estimate of $\nabla F$ is not needed).
It follows from Propositions \ref{1st-LED}, \ref{weak} and \ref{ann-18}
together with \eqref{remainder-est} that
\begin{equation*}
\begin{split}
&\|T(t,\tau)F(\tau)\|_{W^{1,q}(D_R)}
\leq C(m+1)\|f\|_q\;\alpha(\tau), \\
&\|\partial_tT(t,\tau)F(\tau)\|_{W^{-1,q}(D_R)}
\leq C(m+1)\|f\|_q\;\beta(\tau),
\end{split}
\end{equation*}
with
\begin{equation*}
\begin{split}
\alpha(\tau)
&=(t-\tau)^{-1/2}(1+t-\tau)^{-1+\varepsilon}
(\tau-s)^{-1/2}(1+\tau-s)^{-3/2q+1/2}, \\
\beta(\tau)
&=(t-\tau)^{-(1+1/q)/2}(1+t-\tau)^{-1+1/2q+\varepsilon}
(\tau-s)^{-1/2}(1+\tau-s)^{-3/2q+1/2},
\end{split}
\end{equation*}
for $\tau\in (s,t)$.
Then we see that
\[
\int_s^{(s+t)/2}\alpha(\tau)\,d\tau
\leq C(t-s)^{-3/2+\varepsilon}
\left\{
\begin{array}{ll}
1, & q<3/2, \\ 
\log (t-s), &q=3/2, \\ 
(t-s)^{1-3/2q}, \quad &q>3/2,
\end{array}
\right.
\]
as well as
\[
\int_{(s+t)/2}^t\alpha(\tau)\,d\tau
\leq C(t-s)^{-3/2q}
\]
and that the same estimates as above hold for $\beta(\tau)$, too.
We have completed the proof of $\mbox{\eqref{2nd-led}}_1$.

It remains to discuss the adjoint $T(t,s)^*$.
Given $g\in C_{0,\sigma}^\infty(D)$, we describe the solution
$T(t,s)^*g$ in the form
\[
T(t,s)^*g=(1-\phi)U(t,s)^*g+\mathbb B[(U(t,s)^*g)\cdot\nabla\phi]+u(s),
\]
where $\phi$ and $\mathbb B$ are the same as before, while
$U(t,s)^*$ is the evolution operator for the backward problem
\eqref{back-wh}--\eqref{backside-wh} in the whole space
and the first two terms above
possess the decay rate $(t-s)^{-3/2q}$.
Given vector field $\psi\in C_0^\infty(D)^3$,
we know from Lemma \ref{proj} that $P\psi\in Z_q(D)$ for every
$q\in (3/2,\infty)$, which implies \eqref{evo-eqn} with
$f=P\psi$ for such $q$.
With this at hand, as in \cite[(4.17)]{Hi18}, we utilize
\eqref{skew} and \eqref{duality} to compute
\begin{equation*}
\begin{split}
&\quad \partial_\tau\langle P\psi, T(\tau,s)^*u(\tau)\rangle_D \\
&=\partial_\tau\langle T(\tau,s)P\psi, u(\tau)\rangle_D \\
&=\langle T(\tau,s)P\psi, \partial_\tau u(\tau)\rangle_D
-\langle L_+(\tau)T(\tau,s)P\psi, u(\tau)\rangle_D \\
&=\langle T(\tau,s)P\psi, \partial_\tau u(\tau)-L_-(\tau)u(\tau)\rangle_D.
\end{split}
\end{equation*}
This implies the Duhamel formula in the weak form
\begin{equation}
\langle\psi, u(s)\rangle_D
=\langle\psi, T(t,s)^*\widetilde g\rangle_D
+\int_s^t\langle\psi, T(\tau,s)^*G(\tau)\rangle_D\,d\tau
\label{duha-w1}
\end{equation}
for all $\psi\in C_0^\infty(D)^3$ on account of $Pu(s)=u(s)$.
Here, $\widetilde g=\phi g-\mathbb B[g\cdot\nabla\phi]$ and
\begin{equation*}
\begin{split}
G(y,s)
&=-2\nabla\phi\cdot\nabla U(t,s)^*g
-[\Delta\phi-(\eta(s)+\omega(s)\times y)\cdot\nabla\phi]U(t,s)^*g  \\
&\quad +\mathbb B[(\partial_sU(t,s)^*g)\cdot\nabla\phi]
+\Delta\mathbb B[(U(t,s)^*g)\cdot\nabla\phi] \\ 
&\quad -(\eta(s)+\omega(s)\times y)\cdot\nabla 
\mathbb B[(U(t,s)^*g)\cdot\nabla\phi] \\ 
&\quad +\omega(s)\times\mathbb B[(U(t,s)^*g)\cdot\nabla\phi],
\end{split}
\end{equation*}
both of which are solenoidal.
It follows from \eqref{duha-w1} that
\begin{equation}
\langle\psi, \partial_su(s)\rangle_D
=\langle\psi, \partial_sT(t,s)^*\widetilde g\rangle_D
-\langle\psi, G(s)\rangle_D
+\int_s^t\langle\psi, \partial_sT(\tau,s)^*G(\tau)\rangle_D\,d\tau
\label{duha-w2}
\end{equation}
for all $\psi\in C_0^\infty(D)^3$.
Since we intend to derive estimates over $D_R$, let us consider
the test functions $\psi\in C_0^\infty(D_R)^3$ in
\eqref{duha-w1}--\eqref{duha-w2}.
Suppose $t-s>2$.
We know that $\|G(s)\|_q$ enjoys exactly the same estimate as in
\eqref{remainder-est}, see \cite[(4.16)]{Hi18}.
By use of this combined with Propositions \ref{1st-LED}, \ref{ann-18} and
\begin{equation}
\|\partial_sT(t,s)^*g\|_{W^{-1,q}(D_R)}
\leq C(t-s)^{-(1+1/q)/2}\|g\|_q
\label{weak-adj}
\end{equation}
for all $(t,s)\in\Lambda(\tau_*)$ and $g\in L^q_\sigma(D)$, which
corresponds to \eqref{weak-est} for $T(t,s)$,
we find the desired estimates for
$\|u(s)\|_{q,D_R}$ and 
$\|\partial_su(s)\|_{W^{-1,q}(D_R)}$, in which
computations are essentially the same as those
for the last terms of \eqref{duha-1}--\eqref{duha-2} although we employ
the duality.
One can get the desired estimate of
$\|\nabla u(s)\|_{q,D_R}$ as well by taking test functions
of the form $\psi=\mbox{div $\Psi$}$ with 
$\Psi\in C_0^\infty(D_R)^{3\times 3}$ and then by
adopting the same argument as above after integration by parts
in \eqref{duha-w1}.
The proof is complete.
\end{proof}

As a corollary to Proposition \ref{2nd-LED} as well as
Proposition \ref{weak},
one can derive the following asymptotic 
behavior of the pressures associated with $T(t,s)f$ and $T(t,s)^*g$.
This plays an important role in the next section.
\begin{corollary}
Suppose that $\eta$ and $\omega$ fulfill \eqref{rigid} 
for some $\theta\in (0,1)$.
Let $R\in (R_0+1,\infty)$, where $R_0$ is as in \eqref{obstacle}.
Let $1<q<\infty$.
Given $f\in L^q_\sigma(D)$, we denote by
$p(t)$ the pressure associated with $T(t,s)f$ subject to
$\int_{D_R}p\,dx=0$, which is determined by Proposition \ref{weak}.
Let $\phi\in C_0^\infty(B_R)$ satisfy $\phi=1$  in $B_{R_0+1}$,
and $\mathbb B=\mathbb B_{A_R}$ the Bogovskii operator on the bounded domain 
$A_R=\{R_0<|x|<R\}$, see \eqref{bogov}.
Then, for each $m\in (0,\infty)$, there is a constant $C=C(m,q,R,\theta,D)>0$ 
such that
\begin{equation}
\begin{split}
&\|p(t)\|_{q,D_R}
+\|\mathbb B[(\partial_tT(t,s)f)\cdot\nabla\phi]\|_{q,A_R}  \\
&\leq C\|f\|_q
\left\{
\begin{array}{ll}
(t-s)^{-(1+1/q)/2}, \qquad& 0<t-s\leq 2, \\
(t-s)^{-3/2q}, & t-s>2,
\end{array}
\right.
\end{split}
\label{pressure-est}
\end{equation}
for all $f\in L^q_\sigma(D)$ whenever \eqref{rigid-bound} is satisfied.
Here, the temporal derivative is understood as in Proposition \ref{weak}.
The same assertion holds true for $T(t,s)^*g$ with
$g\in L^q_\sigma(D)$ and the associated pressure as well.
\label{pressure}
\end{corollary}

\begin{proof}
Estimate \eqref{pressure-est} near $t=s$ for the pressure was already
obtained in \eqref{pressure-reg}.
By \eqref{bog-est-2} we have
\begin{equation*}
\begin{split}
\|\mathbb B[(\partial_tT(t,s)f)\cdot\nabla\phi]\|_{q,A_R}
&\leq C\|(\partial_tT(t,s)f)\cdot\nabla\phi\|_{W^{1,q^\prime}(A_R)^*}  \\
&\leq C\|\partial_tT(t,s)f\|_{W^{-1,q}(D_R)},
\end{split}
\end{equation*}
which together with \eqref{pressure-cont} 
implies that
\eqref{pressure-est} follows from
\eqref{2nd-led} for large $(t-s)$ as well as \eqref{weak-est}
for small $(t-s)$.
\end{proof}

Another corollary to Proposition \ref{2nd-LED} is the $L^\infty$-estimate.
\begin{corollary}
Suppose that $\eta$ and $\omega$ fulfill \eqref{rigid} for some 
$\theta\in (0,1)$.
Let $R\in (R_0+1,\infty)$, where $R_0$ is as in \eqref{obstacle}.
Let $1<q<\infty$.
For each $m\in (0,\infty)$, 
there is a constant $C=C(m,q,R,\theta,D)>0$ such that
\begin{equation}
\begin{split}
\|T(t,s)f\|_{\infty,D_R}
&\leq C(t-s)^{-3/2q}\|f\|_q, \\
\|T(t,s)^*g\|_{\infty,D_R}
&\leq C(t-s)^{-3/2q}\|g\|_q,
\end{split}
\label{inf-led}
\end{equation}
for all $(t,s)$ with
\[
t-s>2 \quad\mbox{as well as $0\leq s<t$}
\]
and $f,\, g\in L^q_\sigma(D)$ 
whenever \eqref{rigid-bound} is satisfied.
\label{inf-LED}
\end{corollary}

\begin{proof}
$L^\infty$-estimate follows directly from \eqref{2nd-led} together with the Sobolev embedding when $q>3$.
If $q\leq 3$, then we have
\[
\|T(t,s)f\|_{\infty,D_R}
\leq C(t-s)^{-1/4}\|T((t+s)/2,s)f\|_6
\]
which leads to \eqref{inf-led} by 
the first assertion of Proposition \ref{ann-18}.
\end{proof}

\section{Proof of the main theorems}
\label{proof}

In the final section we complete the proof of the main results
on decay estimates of gradient of the evolution
operator $T(t,s)$ and its adjoint $T(t,s)^*$ as well as $L^\infty$-decay
estimates.
\medskip

\noindent
{\it Proof of Theorem \ref{main}}.
Let $1<q<\infty$ ($q>3/2$ for $L^\infty$-estimates)
and fix $R\in (R_0+1,\infty)$, where $R_0$ is as in \eqref{obstacle}.
It then suffices to prove
\begin{equation}
\begin{split}
&\|\nabla T(t,s)f\|_{q,\mathbb R^3\setminus B_R}
\leq C(t-s)^{-\min\{1/2,\,3/2q\}}\|f\|_q,  \\
&\|T(t,s)f\|_{\infty,\mathbb R^3\setminus B_R}
\leq C(t-s)^{-3/2q}\|f\|_q,
\end{split}
\label{outside}
\end{equation}
and
\begin{equation}
\begin{split}
&\|\nabla T(t,s)^*g\|_{q,\mathbb R^3\setminus B_R}
\leq C(t-s)^{-\min\{1/2,\,3/2q\}}\|g\|_q,  \\
&\|T(t,s)^*g\|_{\infty,\mathbb R^3\setminus B_R}
\leq C(t-s)^{-3/2q}\|g\|_q,
\end{split}
\label{outside-adj}
\end{equation}
for all $(t,s)$ with $t-s>2$ as well as $0\leq s<t$
and $f,\, g\in C^\infty_{0,\sigma}(D)$.
From this combined with \eqref{2nd-led}, \eqref{inf-led}, 
Proposition \ref{ann-18} 
and the semigroup property we conclude Theorem \ref{main}.
Note that \eqref{L-inf} for $t-s\leq 2$ follows from Proposition \ref{ann-18}
together with an embedding relation and 
that \eqref{L-inf} for $q>3/2$ yields \eqref{L-inf}
even for $q\leq 3/2$ 
on account of the semigroup property and \eqref{LqLr-noch}.

Let us take a cut-off function $\phi\in C_0^\infty(B_R)$
and the Bogovskii operator $\mathbb B=\mathbb B_{A_R}$
as in Corollary \ref{pressure}.
Given $f\in C^\infty_{0,\sigma}(D)$, we denote by 
$p(t)$ the pressure associated with the velocity $T(t,s)f$ such that
$\int_{D_R}p\,dx=0$.
Set
\begin{equation}
v(t)=(1-\phi)T(t,s)f+\mathbb B[(T(t,s)f)\cdot\nabla\phi], \qquad
p_v(t)=(1-\phi)p(t).
\label{cut-ext}
\end{equation}
Since $v(t)=T(t,s)f$ in $\mathbb R^3\setminus B_R$, let us consider
$\|\nabla v(t)\|_{q,\mathbb R^3}$ and $\|v(t)\|_{\infty,\mathbb R^3}$ by using
\begin{equation}
v(t)=U(t,s)\widetilde f+\int_s^t U(t,\tau)P_{\mathbb R^3}H(\tau)\,d\tau
\label{duha-whole}
\end{equation}
where
$P_{\mathbb R^3}=I+{\cal R}\otimes {\cal R}$ is the Fujita-Kato projection
in the whole space,
$\widetilde f=(1-\phi)f+\mathbb B[f\cdot\nabla\phi]
\in C_{0,\sigma}^\infty(\mathbb R^3)$
and
\begin{equation*}
\begin{split}
H(x,t)
&=2\nabla\phi\cdot\nabla T(t,s)f
+\{\Delta\phi+(\eta+\omega\times x)\cdot\nabla\phi\}T(t,s)f   \\
&\quad -\Delta\mathbb B[(T(t,s)f)\cdot\nabla\phi]
-(\eta+\omega\times x)\cdot\nabla\mathbb B[(T(t,s)f)\cdot\nabla\phi]  \\
&\quad  +\omega\times \mathbb B[(T(t,s)f)\cdot\nabla\phi]   \\
&\quad +\mathbb B[(\partial_tT(t,s)f)\cdot\nabla\phi]-(\nabla\phi)p.
\end{split}
\end{equation*}
Among several terms of which $H$ consists,
the last two terms are always delicate in cut-off procedures, but we have
Corollary \ref{pressure} and that is why we have made effort to analyze
$\partial_tT(t,s)$ in Propositions \ref{1st-LED} and \ref{2nd-LED},
while the other terms are harmless.
Clearly, $H=0$ for $|x|\geq R$, and it is seen 
from \eqref{2nd-led}, \eqref{pressure-est} 
and the second assertion of Proposition \ref{ann-18} that
\begin{equation}
\|H(t)\|_{r,\mathbb R^3}\leq C(m+1)\|f\|_q
\left\{
\begin{array}{ll}
(t-s)^{-(1+1/q)/2}, \quad &0<t-s\leq 2, \\
(t-s)^{-3/2q}, & t-s>2,
\end{array}
\right.
\label{G-est}
\end{equation}
for every $r\in (1,q]$.

Suppose $t-s>2$.
By \eqref{LqLr-whole} 
the first term $U(t,s)\widetilde f$ of 
\eqref{duha-whole} satisfies the desired estimate.
Let us consider the second term of \eqref{duha-whole}.
To this end, we combine \eqref{G-est} with \eqref{LqLr-whole} to observe
\begin{equation*}
\begin{split}
\|\nabla U(t,\tau)P_{\mathbb R^3}H(\tau)\|_{q,\mathbb R^3}
&\leq C(m+1)\|f\|_q \;\widetilde\alpha(\tau), \\
\|U(t,\tau)P_{\mathbb R^3}H(\tau)\|_{\infty,\mathbb R^3}
&\leq C(m+1)\|f\|_q\;\widetilde\beta(\tau),
\end{split}
\end{equation*}
with
\begin{equation*}
\begin{split}
&\widetilde\alpha(\tau)=
(t-\tau)^{-1/2}(1+t-\tau)^{-(3/r-3/q)/2}
(\tau-s)^{-(1+1/q)/2}(1+\tau-s)^{-1/q+1/2}, \\
&\widetilde\beta(\tau)=
(t-\tau)^{-3/2q}(1+t-\tau)^{-(3/r-3/q)/2}
(\tau-s)^{-(1+1/q)/2}(1+\tau-s)^{-1/q+1/2},
\end{split}
\end{equation*}
for $\tau\in (s,t)$,
where $r\in (1,q]$ will be soon chosen appropriately.
Then we have
\[
\int_s^{(s+t)/2}\widetilde\alpha(\tau)\,d\tau
\leq C(t-s)^{-(3/r-3/q)/2-1/2}
\left\{
\begin{array}{ll}
1, & q<3/2, \\
\log (t-s), &q=3/2, \\
(t-s)^{1-3/2q}, \quad &q>3/2.
\end{array}
\right.
\]
By a suitable choice of $r\in (1,q]$, that is,
\[
r=q<3/2, \qquad
r<3/2=q, \qquad
r\leq 3/2<q,
\]
we find
\[
\int_s^{(s+t)/2}\widetilde\alpha(\tau)\,d\tau\leq C(t-s)^{-1/2}
\]
for every $q\in (1,\infty)$.
On the other hand, we observe
\begin{equation*}
\begin{split}
\int_{(s+t)/2}^t \widetilde\alpha(\tau)\,d\tau
&\leq C
\left\{
\begin{array}{ll}
(t-s)^{-3/2q+1/2}, & q\leq 3/2, \\
(t-s)^{-3/2q}, \quad & q>3/2,
\end{array}
\right.  \\
&\leq C
\left\{
\begin{array}{ll}
(t-s)^{-1/2}, & q\leq 3, \\
(t-s)^{-3/2q}, \quad& q>3
\end{array}
\right.
\end{split}
\end{equation*}
where $r$ is chosen to be close to $1$ in such a way that
$1/r>1/q+1/3$ for the case $q>3/2$, 
while it is enough to choose $r=q$ for the other case $q\leq 3/2$.
Summing up all computations above, 
we are led to the gradient estimate in \eqref{outside}.
$L^\infty$-estimate is discussed similarly by use of 
$\widetilde\beta(\tau)$ above as long as $q>3/2$.

Given $g\in C_{0,\sigma}^\infty(D)$,
we next consider $T(t,s)^*g$ together with the associated pressure
$\sigma(s)$ such that
$\int_{D_R}\sigma\,dy=0$, see \eqref{back}.
As in \eqref{cut-ext}, we set
\begin{equation}
u(s)=(1-\phi)T(t,s)^*g+\mathbb B[(T(t,s)^*g)\cdot\nabla\phi], \qquad
\sigma_u(s)=(1-\phi)\sigma(s).
\label{cut-extadj}
\end{equation}
The same argument as above with use of the adjoint
$U(t,s)^*$ being the solution operator to the backward system
\eqref{back-wh}--\eqref{backside-wh} in the whole space
implies \eqref{outside-adj}.
The proof of Theorem \ref{main} is thus complete.
\hfill
$\Box$
\medskip

Let us close the paper with a brief description of the proof
of Theorem \ref{adj}.
\medskip

\noindent
{\it Proof of Theorem \ref{adj}}.
Given $g\in C_{0,\sigma}^\infty(D)$,
as in the last part of the proof of Theorem \ref{main},
we still consider the strong solution $T(t,s)^*g$ and single out
the associated pressure
$\sigma(s)$ satisfying the side condition
$\int_{D_R}\sigma\,dy=0$.

Toward \eqref{adj-Lorentz} with $r=3$ (the most important case for us),
as was discussed in \cite[Section 8]{HiS}
for the autonomous case,
the real interpolation is performed at the level of
\eqref{2nd-led} and \eqref{pressure-est}
for the adjoint $T(t,s)^*$
(as well as \eqref{LqLr-noch} and \eqref{LqLr-grad-noch}) to find that
\begin{equation}
\|\nabla^jT(t,s)^*g\|_{L^{q,\rho}(D_R)}
\leq C\|g\|_{q,\rho}
\left\{
\begin{array}{ll}
(t-s)^{-j/2}, &0<t-s\leq 2, \\
(t-s)^{-3/2q}, \qquad &t-s>2,
\end{array}
\right.
\label{loc-interpo-1}
\end{equation}
with $j=0,\, 1$ and that
\begin{equation}
\begin{split}
&\|\sigma(s)\|_{L^{q,\rho}(D_R)}+
\|\mathbb B[(\partial_sT(t,s)^*g)\cdot\nabla\phi]\|_{L^{q,\rho}(A_R)}  \\
&\leq C\|g\|_{q,\rho}
\left\{
\begin{array}{ll}
(t-s)^{-(1+1/q)/2}, \qquad &0<t-s\leq 2, \\
(t-s)^{-3/2q}, &t-s>2,
\end{array}
\right.
\end{split}
\label{loc-interpo-2}
\end{equation}
where $1<q<\infty$, $1\leq\rho\leq\infty$ and $g\in L^{q,\rho}_\sigma(D)$.
We then proceed to the final step in this section to obtain
$\mbox{\eqref{outside-adj}}_1$ in which $L^q$-norm is now replaced by
$L^{q,\rho}$-norm.
To this end, we consider $u(s)$ given by \eqref{cut-extadj} and
have only to estimate
$\|\nabla u(s)\|_{L^{q,\rho}(\mathbb R^3)}$
by making use of $L^{q,\rho}$-$L^{r,\rho}$ estimates of
$\nabla U(t,s)^*$ and the estimate of the Bogovskii operator
$\mathbb B=\mathbb B_{A_R}$ in $L^{q,\rho}(A_R)$,
which follows from \eqref{bog-est-1} by interpolation,
as well as \eqref{loc-interpo-1}--\eqref{loc-interpo-2}.
The argument ends up with continuity and
that is why the case $\rho=\infty$ is missing in \eqref{adj-Lorentz}.

Finally, following the art developed by Yamazaki \cite{Y},
we perform the real interpolation for the sublinear operator:
$g\mapsto \|\nabla T(t,\cdot)^*g\|_{r,1}$
(for fixed $t>0$ and $r\in (3/2,3]$) to conclude \eqref{adj-int}.
The proof is complete.
\hfill
$\Box$


\begin{thebibliography}{999}
\footnotesize

\bibitem{BL}
J. Bergh and J. L\"ofstr\"om,
{\it Interpolation Spaces},
Springer, Berlin, 1976.

\bibitem{B}
M. E. Bogovski\u\i,
Solution of the first boundary value problem for the equation of
continuity of an incompressible medium,
{\it Soviet Math. Dokl.}
{\bf 20} (1979), 1094--1098.

\bibitem{BS}
W. Borchers and H. Sohr,
On the equations $\mbox{rot $v$}=g$ and $\mbox{div $u$}=f$
with zero boundary conditions,
{\it Hokkaido Math. J.}
{\bf 19} (1990), 67--87.

\bibitem{CMi}
Z.M. Chen and T. Miyakawa,
Decay properties of weak solutions to a perturbed Navier-Stokes system in
$\mathbb R^n$,
{\it Adv. Math. Sci. Appl.}
{\bf 7} (1997), 741--770.

\bibitem{DS-1}
W. Dan and Y. Shibata,
On the $L_q$-$L_r$ estimates of the Stokes semigroup 
in a two-dimensional exterior domain,
{\it J. Math. Soc. Japan}
{\bf 51} (1999), 181--207.

\bibitem{DS-2}
W. Dan and Y. Shibata,
Remark on the $L_q$-$L_\infty$ estimate of the Stokes semigroup 
in a $2$-dimensional exterior domain,
{\it Pacific J. Math.}
{\bf 189} (1999), 223--239.

\bibitem{ES04}
Y. Enomoto and Y. Shibata,
Local energy decay of solutions to the Oseen equation in the exterior domains,
{\it Indiana Univ. Math. J.}
{\bf 53} (2004), 1291--1330.

\bibitem{ES05}
Y. Enomoto and Y. Shibata,
On the rate of decay of the Oseen semigroup in exterior domains and its
applications to the Navier-Stokes equation,
{\it J. Math. Fluid Mech.}
{\bf 7} (2005), 339--367.

\bibitem{FGK}
R. Farwig, G.P. Galdi and M. Kyed,
Asymptotic structure of a Leray solution to the Navier-Stokes flow
around a rotating body,
{\em Pacific. J. Math.} {\bf 253} (2011), 367--382.

\bibitem{FH11}
R. Farwig and T. Hishida,
Leading term at infinity of steady Navier-Stokes flow around a rotating
obstacle,
{\em Math, Nachr.} {\bf 284} (2011), 2065--2077.

\bibitem{FN10}
R. Farwig and J. Neustupa,
Spectral properties in $L^q$ of an Oseen operator modelling
fluid flow past a rotating body,
{\em Tohoku Math. J.} {\bf 62} (2010), 287--309.

\bibitem{FaS}
R. Farwig and H. Sohr,
Weighted $L^q$-theory for the Stokes resolvent in exterior domains,
{\it J. Math. Soc. Japan}
{\bf 49} (1997), 251--288.

\bibitem{Fi65}
R. Finn,
Stationary solutions of the Navier-Stokes equations,
{\it Proc, Symp. Appl. Math.} 
{\bf 17} (1965), 121--153.

\bibitem{Fu68}
D. Fujiwara,
$L^p$-theory for characterizing the domain of the fractional powers
of $-\Delta$ in the half space,
{\it J. Fac. Sci. Univ. Tokyo Sect.IA Math.} 
{\bf 15} (1968), 169--177.

\bibitem{FM}
D. Fujiwara and H. Morimoto,
An $L_r$-theorem of the Helmholtz decomposition of vector fields,
{\it J. Fac. Sci. Univ. Tokyo Sect. IA Math.}
{\bf 24} (1977), 685--700.

\bibitem{Ga02}
G. P. Galdi,
On the motion of a rigid body in a viscous liquid:
a mathematical analysis with applications,
{\it Handbook of Mathematical Fluid Dynamics}, Vol. I,
653--791,
North-Holland, Amsterdam, 2002.

\bibitem{Ga-b}
G. P. Galdi,
{\em An Introduction to the Mathematical Theory of the Navier-Stokes 
Equations, Steady-State Problems},
Second Edition, Springer, 2011.

\bibitem{Ga-new}
G. P. Galdi,
Viscous flow past a body translating by time-periodic
motion with zero average,
Preprint.

\bibitem{GHS}
G.P. Galdi, J.G. Heywood and Y. Shibata,
On the global existence and convergence to steady state of Navier-Stokes
flow past an obstacle that is started from rest,
{\it Arch. Rational Mech. Anal.}
{\bf 138} (1997), 307--318.

\bibitem{GHa}
M. Geissert and T. Hansel,
A non-autonomous model problem for the Oseen-Navier-Stokes
flow with rotating effect,
{\it J. Math. Soc. Japan}
{\bf 63} (2011), 1027--1037.

\bibitem{GHH}
M. Geissert, H. Heck and M. Hieber,
$L^p$-theory of the Navier-Stokes flow in the exterior
of a moving or rotating obstacle,
{\it J. Reine Angew. Math.}
{\bf 596} (2006), 45--62.

\bibitem{GHH-b}
M. Geissert, H. Heck and M. Hieber,
On the equation div $u=g$ and Bogovski\u\i's operator in Sobolev spaces
of negative order,
{\it Operator Theory: Advances and Applications}
{\bf 168} (2006), 113--121.

\bibitem{Gi85}
Y. Giga,
Domains of fractional powers of the Stokes operator in $L_r$ spaces,
{\it Arch. Rational Mech. Anal.}
{\bf 89} (1985), 251--265.

\bibitem{Ha}
T. Hansel,
On the Navier-Stokes equations with rotating effect and prescribed
outflow velocity,
{\it J. Math. Fluid Mech.}
{\bf 13} (2011), 405--419.

\bibitem{HR11}
T. Hansel and A. Rhandi,
Non-autonomous Ornstein-Uhlenbeck equations in exterior domains,
{\it Adv. Differential Equations}
{\bf 16} (2011), 201--220.

\bibitem{HR14}
T. Hansel and A. Rhandi,
The Oseen-Navier-Stokes flow in the exterior of a rotating obstacle:
the non-autonomous case,
{\it J. Reine Angew. Math.}
{\bf 694} (2014), 1--26.

\bibitem{Hi99}
T. Hishida,
An existence theorem for the Navier-Stokes flow in the exterior
of a rotating obstacle,
{\it Arch. Rational Mech. Anal.}
{\bf 150} (1999), 307--348.

\bibitem{Hi04}
T. Hishida,
The nonstationary Stokes and Navier-Stokes flows through an aperture,
{\it Contributions to Current Challenges in Mathematical Fluid Mechanics}
(eds. G.P. Galdi, J.G. Heywood and R. Rannacher),
79--123,
{\it Adv. Math. Fluid Mech}., Birkh\"auser, Basel, 2004.

\bibitem{Hi11}
T. Hishida,
On the relation between the large time behavior of the Stokes
semigroup and the decay of steady Stokes flow at infinity,
{\it Parabolic Problems: The Herbert Amann Festschrift},
{\it Progress in Nonlinear Differential Equations and Their Applications}
{\bf 80}, 343--355, Springer, Berlin, 2011.

\bibitem{Hi13}
T. Hishida,
Mathematical analysis of the equations for incompressible viscous
fluid around a rotating obstacle,
{\it Sugaku Expositions}
{\bf 26} (2013), 149--179.

\bibitem{Hi16}
T. Hishida,
$L^q$-$L^r$ estimate of the Oseen flow in plane exterior domains,
{\it J. Math. Soc. Japan}
{\bf 68} (2016), 295--346.

\bibitem{Hi-hb}
T. Hishida,
Stationary Navier-Stokes flow in exterior domains and Landau solutions,
{\it Handbook of Mathematical Analysis in Mechanics of Viscous Fluids}
(eds. Y. Giga and A. Novotny),
299--339, Springer, 2018.

\bibitem{Hi18}
T. Hishida,
Large time behavior of a generalized Oseen evolution operator,
with applications to the Navier-Stokes flow past a rotating obstacle,
{\it Math. Ann.}
{\bf 372} (2018), 915--949.

\bibitem{Hi19}
T. Hishida,
$L^q$-$L^r$ estimate of
a generalized Oseen evolution operator,
with applications to the Navier-Stokes flow past a rotating obstacle,
{\it RIMS Kokyuroku}
{\bf 2107} (2019), 139--152.

\bibitem{HiM}
T. Hishida and P. Maremonti,
Navier-Stokes flow past a rigid body:
attainability of steady solutions as limits of unsteady weak
solutions, starting and landing cases,
{\it J. Math. Fluid Mech.}
{\bf 20} (2018), 771--800.

\bibitem{HiS}
T. Hishida and Y. Shibata,
$L_p$-$L_q$ estimate of the Stokes operator and Navier-Stokes flows in the
exterior of a rotating obstacle,
{\it Arch. Rational Mech. Anal.}
{\bf 193} (2009), 339--421.

\bibitem{I}
H. Iwashita,
$L_q$-$L_r$ estimates for solutions of the nonstationary Stokes equations
in an exterior domain and the Navier-Stokes initial value problems 
in $L_q$ spaces,
{\it Math. Ann.}
{\bf 285} (1989), 265--288.

\bibitem{Ka}
T. Kato,
Strong $L^p$ solutions of the Navier-Stokes equation in $\mathbb R^m$,
with applications to weak solutions,
{\em Math. Z.} {\bf 187} (1984), 471--480.

\bibitem{KS}
T. Kobayashi and Y. Shibata,
On the Oseen equation in the three dimensional exterior domains,
{\it Math. Ann.}
{\bf 310} (1998), 1--45.

\bibitem{KSv}
A. Korolev and V. \v Sverak,
On the large-distance asymptotics of steady state solutions of
the Navier-Stokes equations in 3D exterior domains,
{\it Ann. Inst. Henri Poincare, Anal. Non Lineaire}
{\bf 28} (2011) 303--313.

\bibitem{Mae}
Y. Maekawa,
On local energy decay estimate of the Oseen semigroup in two dimensions 
and its application,
{\it J. Inst. Math. Jussieu},
published online,
doi: 10.1017/s1474748019000355

\bibitem{MSol}
P. Maremonti and V.A. Solonnikov,
On nonstationary Stokes problems in exterior domains,
{\it Ann. Sc. Norm. Sup. Pisa}
{\bf 24} (1997), 395--449.

\bibitem{Mi}
T. Miyakawa,
On nonstationary solutions of the Navier-Stokes equations 
in an exterior domain,
{\it Hiroshima Math. J.}
{\bf 12} (1982), 115--140.

\bibitem{Shi08}
Y. Shibata,
On the Oseen semigroup with rotating effect,
{\it Functional Analysis and Evolution Equations, The G\"unter Lumer Volume},
595--611, Birkh\"auser, Basel, 2008.

\bibitem{Shi10}
Y. Shibata,
On a $C^0$ semigroup associated with a modified Oseen equation with rotating 
effect,
{\it Adv. Math. Fluid Mech.}, 
513--551, Springer, Berlin, 2010.

\bibitem{Shi-aa}
Y. Shibata,
On the $L_p$-$L_q$ decay estimate for the Stokes equations with
free boundary conditions in an exterior domain,
{\it Asymptot. Anal.}
{\bf 107} (2018), 33--72.

\bibitem{SiS}
C.G. Simader and H. Sohr,
A new approach to the Helmholtz decomposition and the Neumann problem
in $L^q$-spaces for bounded and exterior domains,
{\it Mathematical Problems Relating to the Navier-Stokes Equations}
(eds. G.P. Galdi), 1--35,
{\em Ser. Adv. Math. Appl. Sci.} {\bf 11}, World Sci. Publ.,
River Edge, NJ, 1992.

\bibitem{St}
E.M. Stein,
{\it Harmonic Analysis:
real-variable methods, orthogonality, and oscillatory integrals},
Princeton Univ. Press, New Jersey, 1993.

\bibitem{Ta-new}
T. Takahashi,
Attainability of steady Navier-Stokes flow around a rigid body
rotating from rest,
In preparation.

\bibitem{T}
H. Tanabe,
{\it Equations of Evolution}, 
Pitman, London, 1979.

\bibitem{To}
A. Torchinsky,
{\it Real-Variable Methods in Harmonic Analysis},
Academic Press, 1986.

\bibitem{Y}
M. Yamazaki,
The Navier-Stokes equations in the weak-$L^n$ space with
time-dependent external force,
{\it Math. Ann.}
{\bf 317} (2000), 635--675.

\end{thebibliography}
\end{document}